\documentclass[12pt,leqno]{article}
\usepackage{amsmath, amsthm, amsfonts, amssymb, color}
\usepackage{mathrsfs}
\usepackage[notcite,notref]{showkeys}

\newtheorem{thm}{Theorem}[section]
\newtheorem{cor}[thm]{Corollary}
\newtheorem{lem}[thm]{Lemma}
\newtheorem{prp}[thm]{Proposition}

\newtheorem{rem}[thm]{Remark}

\theoremstyle{definition}

\newcommand{\scr}[1]{\mathscr #1}
\definecolor{wco}{rgb}{0.5,0.2,0.3}

\numberwithin{equation}{section}

\newcommand{\ua}{\uparrow}

\title{{\bf Navier-Stokes equation
and forward-backward stochastic differential system in the Besov spaces}}

\author{
{\bf   Xin Chen$^{a)}$, Ana Bela Cruzeiro$^{a),b)}$, Zhongmin Qian$^{c)}$}\\
\footnotesize{$^{a)}$ Grupo de F\'isica Matem\'atica, Universidade de
Lisboa, Av. Prof. Gama Pinto 2, 1649-003  Lisbon, Portugal}\\
\footnotesize{$^{b)}$ Dep. Matem\'atica IST (TUL), Av. Rovisco Pais,  1049-001  Lisbon, Portugal}\\
\footnotesize{$^{c)}$ Mathematical Institute,
University of Oxford
24-29 St Giles', Oxford OX1 3LB, UK}\\
\footnotesize{chenxin\_217@hotmail.com, abcruz@math.ist.utl.pt, Zhongmin.Qian@maths.ox.ac.uk}}

\begin{document}

\allowdisplaybreaks

\def\R{\mathbb R} \def\Z{\mathbb Z} \def\ff{\frac} \def\ss{\sqrt}
\def\dd{\delta} \def\DD{\Delta} \def\vv{\varepsilon} \def\rr{\rho}
\def\<{\langle} \def\>{\rangle} \def\GG{\Gamma} \def\gg{\gamma}
\def\ll{\lambda} \def\LL{\Lambda} \def\nn{\nabla} \def\pp{\partial}
\def\d{\text{\rm{d}}} \def\loc{\text{\rm{loc}}} \def\bb{\beta} \def\aa{\alpha} \def\D{\mathbb D}
\def\E{\mathbb E}
\def\O{\mathbb O}
\newcommand{\Ex}{{\bf E}}
\def\si{\sigma} \def\ess{\text{\rm{ess}}}
\def\beg{\begin} \def\beq{\beg}  \def\F{\scr F}
\def\Ric{\text{\rm{Ric}}} \def\Hess{\text{\rm{Hess}}}\def\B{\scr B}
\def\e{\mathbb E} \def\ua{\underline a} \def\OO{\Omega} \def\b{\mathbf b}
\def\oo{\omega}     \def\tt{\tilde} \def\Ric{\text{\rm{Ric}}}
\def\cut{\text{\rm{cut}}}
\def\fff{f(x_1)\dots f(x_n)} \def\ifm{I_m(g^{\bigotimes m})} \def\ee{\varepsilon}
\def\C{\text{curl}}
\def\B{\scr B}
\def\M{\tilde{M}}\def\ll{\lambda}
\def\X{\scr X}
\def\T{\mathbb T}
\def\A{\scr A}
\def\LL{\scr L}
\def\gap{\mathbf{gap}}
\def\div{\text{\rm div}}
\def\dist{\text{\rm dist}}
\def\cut{\text{\rm cut}}
\def\supp{\text{\rm supp}}
\def\Var{\text{\rm Var}}
\def\p{\mathbf{p}}
\def\Cov{\text{\rm Cov}}
\def\Cut{\text{\rm Cut}}
\def\le{\leqslant}
\def\I{\scr I}
\def\coth{\text{\rm coth}}
\def\Dom{\text{\rm Dom}}
\def\Cap{\text{\rm Cap}}
\def\Ent{\text{\rm Ent}}
\def\sect{\text{\rm sect}}\def\H{\mathbb H}
\def\g{\tilde {g}}
\def\u{\tilde{u}}
\def\c{\tilde{\theta}}\def\w{\tilde{\omega}}
\def\om{\tilde{\Omega}}\def\v{\varepsilon}
\def\U{\tilde{U}}
\def\a{\tilde \alpha}\def\b{\tilde \beta}
\def\p{\tilde \rho}\def\s{\tilde \sigma}
\def\S{\scr S}

\maketitle

\rm

\vskip 10mm

\begin{abstract}
The Navier-Stokes equation on $\R^d$ ($d \geqslant 3$) formulated on  Besov spaces is considered. Using
a stochastic forward-backward differential system,  the local existence of a unique solution in
$B_{p,p}^{r}$, with $r>1+\frac{d}{p}$ is obtained.
We also show the convergence to solution of the Euler equation when the
viscosity tends to zero. Moreover, we prove the
local existence of a unique solution in
$B_{p,q}^{r}$, with $p>1$, $1\le q < \infty$, $r>\max(1,\frac{d}{p})$; here the maximal time interval
depends on the viscosity $\nu$.
\end{abstract}

\vskip 3mm

\bf Contents
\rm

1.  Introduction

2.  Notations and preliminary results

3.  The local existence theorem in  $B_{p,p}^{r}(\R^d;\R^d)$

4.  The limit to the Euler Equation as $\nu \rightarrow 0$

5.  The local existence theorem in $B_{p,q}^{r}(\R^d;\R^d)$

\section{Introduction}

The motion and evolution of the velocity field in an incompressible fluid
can be described by the following Navier-Stokes equation in $[0,T]\times \R^d$ ($d\geqslant 2$),
\begin{equation}\label{e1}
\begin{cases}
&\frac{\partial u}{\partial t}+u\cdot \nabla u=\nu \Delta u-\nabla p,\\
&\nabla \cdot u=0,\ \ \ u(0)=u_0, \ \ \ t\in[0,T],
\end{cases}
\end{equation}
where  $u:
[0,T]\times \R^d \rightarrow \R^d$ represents the velocity field and
$\nabla \cdot u$ denotes its divergence,
$p:
[0,T]\times \R^d \rightarrow \R$ denotes the pressure, and
$\nu>0$ is the viscosity.
In particular $p$ is a function satisfying the following equation:
\begin{equation}\label{e2}
\Delta p(t,x)=-\sum_{i,j=1}^d\partial_i u^j(t,x)\partial_j u^i(t,x),\ \ \forall \ t\in[0,T],
\end{equation}
where for the vector filed $u=(u^1,\cdots,u^d)$,
$\partial_i u^j$, $1\le i,j \le d$ denotes the partial derivative with respect to the $i$-th variable for the
$j$-th component of $u$.

The Navier-Stokes equation (\ref{e1}) is and has been for a long time  the subject of many works. We refer for example to the
books \cite{CF}, \cite{MB}, \cite{Te} and the references therein.
More recently the Navier-Stokes equation  has been studied  using stochastic methods.
In \cite{LS} Y. Le Jan and A. S. Sznitman used a branching process in Fourier space to show the local existence and uniqueness of
solutions on $\R^d$. In \cite{CI}, P. Constantin and G. Iyer  obtained a stochastic
representation of  (\ref{e1}) by using the associated stochastic Lagrangian paths; in particular
the solution
of (\ref{e1}) is seen to be equivalent to the solution of a stochastic-functional system. And in \cite{I}, G. Iyer
derived the
 local existence of a unique solution in H\"older spaces on the torus by proving the
corresponding result for the equivalent stochastic-functional system. Moreover,
a backward stochastic Lagrangian path was used by X. C. Zhang in \cite{Z} to give a stochastic representation for
the backward version of (\ref{e1}), and the local existence and uniqueness of the solution in Sobolev space
were  also obtained based on such representation.
In \cite{B}, B. Busnello  proved the existence and uniqueness of the solution on $\R^2$ by analyzing
the corresponding  vorticity equation
and the Biot-Savart law. Based on such formulation, in \cite{BFR}  local existence and  uniqueness
of the solution in H\"older space on $\R^3$ was shown by B. Busnello,
F. Flandoli and M. Romito via a generalized Feynman-Kac formula.
In \cite{AB} S. Albeverio and Y. Belopolskaya obtained the local existence of a unique solution in H\"older  space
on $\R^d$  via a semi-group expression for this solution.
Moreover, the global existence of the solution on the torus has been studied in \cite{I1} and \cite{Z}.
Recently, in \cite{CQ}, A. B. Cruzeiro and Z. M. Qian proved  global existence of a unique solution in Sobolev spaces on the
two dimensional torus using  the vorticity equation and the  associated backward SDE.

On the other hand, in
\cite{CC}, a characterization of the solution for Navier-Stokes equation was derived by
F. Cipriano and A. B. Cruzeiro through
some stochastic variational principle, which was formulated on the group of  volume preserving diffeomorphisms.
We refer to \cite{ACC} for the generalization of this approach to general
Lie groups. In \cite{CS}, A. B. Cruzeiro and E. Shamarova established
an equivalence between the solution of (\ref{e1}) and the solution
of a forward-backward stochastic differential equation on the space of volume preserving
maps.

The purpose of this article is to study
the local existence of a unique solution for Navier-Stokes equation in $\R^d$ with $d\geqslant 3$
in Besov spaces via the forward-backward stochastic differential systems (\ref{e3}) and (\ref{e28a}).
Our methods are partially inspired by those of \cite{CI}, \cite{CS} and \cite{CQ}. More precisely, inspired by
\cite{CS}, we will prove the  local existence of a unique solution for (\ref{e1}) by proving the
corresponding property for
an equivalent forward-backward stochastic differential system.
As in \cite{CI}, we use a stochastic Lagrangian path (forward-equation) which is independent of the viscosity 
$\nu$. 
Inspired by \cite{CQ}, we can also choose a different forward equation in the stochastic functional system.


A certain  linear backward SDE was first introduced by J. M. Bismut in \cite{Bis1}.
In \cite{PP1}, E. Pardoux and S. Peng made the important observation that
there exists a unique solution for a general (non-linear) backward SDE. In
\cite{PP}, the connection between   forward-backward SDEs and
quasi-linear PDEs was established by E. Pardoux and S. Peng, which can be viewed as
a generalization of Feynman-Kac's formula; see also \cite{MY} and the
reference therein for an introduction of the forward-backward SDEs
with more general forms.

There are some results on the existence of solutions in the Besov spaces for (\ref{e1}), most of
them  were proved via analytic methods. For example, on $\R^3$, the existence of a strong solution in
(homogeneous Besov space) $\dot{B}_{p,q}^{\frac{3}{p}-1}$
with $1\le p<\infty$, $1\le q \le \infty$
was shown in \cite{C}, and such result was extended to the case of $p=\infty$ in \cite{BBT}, see also \cite{L}.
For the definition of  strong solutions, we refer to \cite{FK}.
 On the other hand,
in \cite{Ca}, the existence of a unique strong solution in some subspace of $\dot{B}_{p,\infty}^{\frac{3}{p}-1}$ for small initial data on
$\R^3$ for $3<p\le 6$ was obtained.
In \cite{W}, in the following three spaces (or their subspaces): (1) $B_{p,\infty}^r$ with $1\le p \le 2$, $r>1$, $r>\frac{d}{p}-1$;
(2) $B_{2,1}^r$ with $r>1$, $r\geqslant \frac{d}{2}-1$; (3) $B_{2,q}^r$ with
$1<q<\infty$, $r>1$, $r> 1+\frac{d}{2}-\frac{2}{q}$, the
existence of a unique strong solution for small initial data was shown.

In fact, in order to get the results above, the viscosity coefficient
$\nu$ needs to be strictly positive; in our paper we can show the local existence of a unique solution
in $B_{p,p}^{r}$ with $p>1$, $r>1+\frac{d}{p}$, and the maximal time interval is independent of $\nu$, which implies  that
such result can be applied to the Euler equation.
In the proof, we use the
same Lagrangian path (forward equation) as the one in \cite{CI}
(also the same as  in \cite{AB},\cite{BFR},\cite{CS}, \cite{Z}).
We also show a result about the convergence of
the solution as $\nu \rightarrow 0$. More generally, in the spaces $B_{p,q}^{r}$ with $p>1$, $1\le q <\infty$,
$r>\max(1,\frac{d}{p})$ (or in subspaces of these spaces), we also  prove  local existence of a unique solution. Here we adopt a different Lagrangian path from the one in \cite{CS} and, in this case, however,
the  maximal time interval for the solution depends on $\nu$.

During the finalization of this paper we found a recent work \cite{DQT} by
F. Delbaen, J. N. Qiu and S. J. Tang, where a forward-backward stochastic functional system different from
ours was introduced. Moreover, the local existence of a unique solution of
(\ref{e1}) in Sobolev space was derived by studying the corresponding property of such system, and the maximal
time interval depends on the viscosity $\nu$.

This article is organized as follows: in Section 2 we give a brief description of
the framework, and prove some Lemmas that will be needed later; in Section 3 we present
the unique local existence of the solution in  $B_{p,p}^{r}(\R^d;\R^d)$ with $p>1$, $r>1+\frac{d}{p}$
for (\ref{e1}) and in Section 4 we study
the limit behaviour of the Navier-Stokes solution as the viscosity $\nu$ tends to $0$. In Section 5 we prove
the unique local existence of the solution in  $B_{p,q}^{r}(\R^d;\R^d)$ with $p>1$, $1\le q< \infty$,
$r>\max(1,\frac{d}{p})$ for equation (\ref{e1}).

\section{Notations and preliminary results}
Throughout this paper we consider the Navier-Stokes equation in $\R^d$ with $d \geqslant 3$. Let $C_c^{\infty}(\R^d)$ denote
the set of smooth functions on Euclidean space $\R^d$ which have
compact supports and $C_c^{\infty}(\R^d;\R^d)$ denote the set of smooth vector fields in
$\R^d$
with compact supports. Analogously $C_b^{\infty}(\R^d)$ stands for the set of smooth bounded functions.
For a vector field $v=(v^1,\cdots,v^d)$ in $\R^d$, the divergence of $v$ is denoted by $\nabla \cdot v$, and
$\partial_i v^j$, $1\le i,j \le d$ stands for the partial derivative with respect to the $i$-th variable of the $j$-th component of $v$.
If $\mathbb{N}$ is the set of natural numbers, for every non-negative $k \in \mathbb{N}$ and every real number $p>1$, the Sobolev space
$W^{k,p}(\R^d;\R^d)$ of vector fields is the completion of
$C_c^{\infty}(\R^d;\R^d)$ under the following norm,
\begin{equation*}
||v||_{W^{k,p}}:=\sum_{i=0}^k ||\nabla^i v||_{L^p},
\end{equation*}
where $\nabla^i$ is the $i$-th differential of $v$, and $||.||_{L^p}$ denotes the $L^p$ norm (with respect to Lebesgue measure).
The Sobolev space $W^{k,p}(\R^d)$ of functions can be defined in the same way.

As in \cite{S} and \cite[Section 1.2]{T1}, we introduce the Besov space $B_{p,q}^{k+\alpha}(\R^d;\R^d)$ in
$\R^d$, where $k$ is a non-negative integer, $0<\alpha<1$, $1<p<\infty, 1\le q \le \infty$.

If $1\le q <\infty$,
\begin{equation*}
B_{p,q}^{k+\alpha}(\R^d;\R^d):=\Big\{v \in W^{k,p}(\R^d;\R^d),\ \ \
\int_{\R^d}\frac{||\nabla^k v(\cdot+y)-\nabla^k v(\cdot)||_{L^p}^q}{|y|^{d+\alpha q}}dy<\infty\Big\}.
\end{equation*}
If $q=\infty$,
\begin{equation*}
B_{p,\infty}^{k+\alpha}(\R^d;\R^d):=\Big\{v \in W^{k,p}(\R^d;\R^d),\ \ \
\sup_{|y|>0}\Big\{\frac{||\nabla^k v(\cdot+y)-\nabla^k v(\cdot)||_{L^p}}{|y|^{\alpha }}
\Big\}<\infty\Big\}.
\end{equation*}

For each $v \in  B_{p,q}^{k+\alpha}(\R^d;\R^d)$, we define the following quasi-norm
\begin{equation*}
||v||_{B_{p,q}^{k+\alpha}}:=||v||_{W^{k,p}}+[v]_{B_{p,q}^{k+\alpha}},
\end{equation*}
where
\begin{equation}\label{e0}
[v]_{B_{p,q}^{k+\alpha}}:=\Big(\int_{\R^d}\frac{||\nabla^k v(\cdot+y)-\nabla^k v(\cdot)||_{L^p}^q}{|y|^{d+\alpha q}}dy\Big)^{\frac{1}{q}}
\end{equation}
if $1\le q <\infty$,
and
\begin{equation}\label{e0a}
[v]_{B_{p,\infty}^{k+\alpha}}:=\sup_{|y|>0}\Big\{\frac{||\nabla^k v(\cdot+y)-\nabla^k v(\cdot)||_{L^p}}{|y|^{\alpha }}
\Big\}
\end{equation} if $q=\infty$.

To see the various equivalent definitions of $B_{p,q}^{k+\alpha}(\R^d;\R^d)$, (for example via
the Fourier multiplication)
one can refer to \cite[Chapter 1]{T1} or \cite[Section 2.5]{T}. In particular,
$B_{p,p}^{k+\alpha}(\R^d)$ is the fractional order Sobolev space $W^{k+\alpha,p}(\R^d)$, see
\cite[Chapter 1]{T1}.
\begin{rem}
Note that the space $B_{p,p}^{r}(\R^d)$ (or $B_{p,q}^{r}(\R^d)$)  may not coincide with
 $W^{r,p}(\R^d)$ if $r$ is an integer (see \cite{T}). For this reason and unless particularly clarified, we shall consider that $r$ is not an integer in this paper.
    
\end{rem}

Suppose that $u$ is a solution of (\ref{e1}) in the  time interval
$t \in [0,T]$, which is regular enough (for example $u \in C([0,T]; C_b^{\infty}(\R^d,\R^d))$ ).
Fix a Brownian motion $B_t$ on $\R^d$, for $0\le t \le s \le T$ and let
$X^t_s(x)$ be the unique solution of the following SDE,
\begin{equation*}
\begin{cases}
& dX_s^t(x)=\sqrt{2\nu} dB_s-u(T-s,X_s^t(x))ds\\
& X_t^t(x)=x,
\end{cases}
\end{equation*}
We define $Y^t_s(x):=u(T-s,X^t_s(x))$,
$Z^{t}_s(x):=\nabla u(T-s,X^t_s(x))$. Applying  It\^o's formula directly, we derive the following
(forward-backward) stochastic differential system on function space whose solution is
$(X_s^t(x),Y_s^t(x),Z_s^t(x),u(t,x), p(t,x))$,
\begin{equation}\label{e3}
\begin{cases}
& dX_s^t(x)=\sqrt{2\nu} dB_s-u(T-s,X_s^t(x))ds\\
& dY_s^t(x)=\sqrt{2\nu}Z^{t}_s(x)dB_s+\nabla p(T-s, X_s^t(x))ds\\
& Y_t^t(x)=u(T-t,x),\  \Delta p(t,x)=-\sum_{i,j=1}^3\partial_i u^j(t,x)\partial_j u^i(t,x)\\
& X_t^t(x)=x, Y_T^t(x)=u_0(X^t_T(x)).
\end{cases}
\end{equation}
On the other hand, if
$(X_s^t(x),Y_s^t(x),Z_s^t(x),u(t,x),p(t,x))$ is a solution of (\ref{e3}) and $u$ is regular enough, for example,
$u \in C([0,T];C_b^3(\R^d;\R^d))$, then $(X_s^t(x),Y_s^t(x),Z_s^t(x))$ satisfies a backward SDE where $u$ and $p$ appear in the
coefficients, so by \cite[Theorem 3.2]{PP}, the vector field $u(t,x):=Y_{T-t}^{T-t}(x)$ satisfies
equation (\ref{e1}) for $t\in [0,T]$. In particular, we can show that $\nabla \cdot u(t,x)=0$ due to the expression
of $p(t,x)$ in (\ref{e3}).
We refer to \cite{CS} for the formulation
of such system on the space of diffeomorphsim group.
Inspired by this argument, we will construct a solution of the system (\ref{e3}) in Besov space, which is still a solution for
Navier-Stokes equation (\ref{e1}) in such function space.

Let us introduce the following notation:
\begin{equation}\label{e5a}
\begin{split}
&F_v:=\nabla N G_v,\ \ G_v:=\sum_{i,j=1}^d\partial_i v^j\partial_j v^i,
\\
& Nf(x):=C(d)\int_{\R^d}\frac{f(y)}{|x-y|^{d-2}}dy,\ \ \forall f \in C_c^{\infty}(\R^d),
\end{split}
\end{equation}
where $N$ is the Newton's potential in $\R^d$, $C(d)$ is a constant depending on $d$, and $\Delta N f(x)= f(x)$ for every
$f \in C_c^{\infty}(\R^d)$. Furthermore   potential theory (see \cite{S}) assures that $Nf$ is well defined for
every $f\in L^{p'}(\R^d)$ with $1<p'<\frac{d}{2}$.

Given a
$u_0 \in C_c^{\infty}(\R^d;\R^d)$, $v \in C([0,T];C_c^{\infty}(\R^d;\R^d))$ with $\nabla \cdot u_0=0$ and
$\nabla \cdot v(t)=0$ for every $t$,
let $(X_s^t,Y_s^t,Z_s^t)$(we omit the index $v$ here) be the unique solution of the following BSDE,
\begin{equation}\label{e5}
\begin{cases}
& dX_s^t(x)=\sqrt{2\nu} dB_s-v(T-s,X_s^t(x))ds\\
& dY_s^t(x)=\sqrt{2\nu}Z^{t}_s(x)dB_s-F_v(T-s, X_s^t(x))ds\\
& X_t^t(x)=x, Y_T^t(x)=u_0(X^t_T(x)),
\end{cases}
\end{equation}
where
$$F_v (t,x):=F_{v(t)}(x)= \nabla N G_{v(t)} (x),~t\in [0,T], x\in \R^d$$
for simplicity.

By the potential theory, $F_v$ is well defined for every
$v \in W^{2,p'}(\R^d)\bigcap W^{2,p}(\R^d)$ with some $1<p'<\frac{d}{2}$ and $p>d$. Following the methods in \cite[Chapter 2]{MB}, see also \cite{S},
we can prove the following Lemma
on the Besov norm bounds of  $F_v$.

In this paper, the constant $C$ may change in different lines in the proofs, but will be independent
of the variables stated in the conclusion.
\begin{lem}\label{l1} Let $d<p<\infty$, $1\le q \le \infty$ and $1<p'<\frac{d}{2}$. Then for all
$v\in \bigcap_{r>1}B_{p,q}^{r}(\R^d ; \R^d)\cap  W^{2,p'}(\R^d ; \R^d)$ satisfying
$\nabla \cdot v=0$,
and for every $0<\alpha<1$,
we have
\begin{equation}\label{e6}
\begin{split}
&||F_v||_{L^p}\le C_1||\nabla v||_{L^{\infty}}||v||_{L^p},\\
&||F_v||_{B^{1+\alpha}_{p,q}}\le C_1||\nabla v||_{L^{\infty}}||v||_{B^{1+\alpha}_{p,q}},\\
&||F_v||_{B^{2+\alpha}_{p,q}}\le C_1||v||_{W^{2,p}}||v||_{B^{2+\alpha}_{p,q}},
\end{split}
\end{equation}
where $C_1$ is a positive constant
independent of $v$ and $p'$.
\end{lem}
\begin{proof}
Since $\nabla \cdot v=0$,
\begin{equation}\label{e6a}
\begin{split}
& G_v(x)=\sum_{i,j}\partial_i v^j(x)\partial_j v^i(x)=
\sum_{i}\partial_i \big(\sum_j v^j(x)\partial_j v^i(x)\big):=
\sum_i \partial_i f_i(x),
\end{split}
\end{equation}
hence $F_v(x)=\nabla N G_v(x)=\sum_i \nabla N \partial_i f_i (x)$. It is easy to check that
$\nabla N \partial_i$ is a singular integral operator as defined in \cite[Chapter 2]{S}, so it is bounded in
$L^p$ space for every $1<p<\infty$ (see \cite[Theorem 3, Chapter 2]{S}), which implies that
\begin{equation*}
||F_v||_{L^p}\le C \sum_i||f_i||_{L^p}\le C||\nabla v||_{L^{\infty}}||v||_{L^p}.
\end{equation*}
Note that $\nabla F_v(x)=\nabla^2 NG_v(x)$ and that

$\nabla^2 N$ is a singular integral operator, we obtain,
\begin{equation*}
||\nabla F_v||_{L^p}\le C ||G_v||_{L^p} \le C||\nabla v||_{L^{\infty}}||\nabla v||_{L^p}.
\end{equation*}
In the same way as above we can show that,
\begin{equation*}
||\nabla^2 F_v||_{L^p}\le C ||\nabla G_v||_{L^p} \le C||\nabla v||_{L^{\infty}}||\nabla^2 v||_{L^p}.
\end{equation*}
For every $y\in \R^d$, let $\tilde G_v^y(x):=G_v(x+y)-G_v(x)$. Then

\begin{equation*}
\begin{split}
& ||\nabla \tilde G_v^y||_{L^p}\\
&\le C\Big(||\nabla v||_{L^{\infty}}||\nabla^2 v(\cdot +y)-
\nabla^2 v(\cdot )||_{L^p}+||\nabla^2 v||_{L^p}||\nabla v(\cdot +y)-
\nabla v(\cdot )||_{L^{\infty}}\Big).
\end{split}
\end{equation*}
Since $N$ is translation invariant,
\begin{equation*}
\nabla^2F_v(x+y)-\nabla^2F_v(x)=\nabla^2 N \big(\nabla \tilde G_v^y\big)(x).
\end{equation*}
Applying singular integral estimates to $\nabla^2 N$ we obtain

\begin{equation}\label{e7aa}
\begin{split}
&||\nabla^2F_v(\cdot +y)-\nabla^2F_v(\cdot )||_{L^p}\le C||\nabla \tilde G_v^y||_{L^p}\\
&\le C\Big(||\nabla v||_{L^{\infty}}||\nabla^2 v(\cdot +y)-
\nabla^2 v(\cdot )||_{L^p}+||\nabla^2 v||_{L^p}||\nabla v(\cdot +y)-
\nabla v(\cdot )||_{L^{\infty}}\Big).
\end{split}
\end{equation}
According to the embedding theorem \cite[Theorem 2.8.1]{T}, since $ v \in B_{p,q}^{2+\alpha}(\R^d;\R^d)$ and $p>d$,
\begin{equation}\label{e7d}
|\nabla v(x)-\nabla v(y)|\le C ||v||_{B_{p,q}^{2+\alpha}}|x-y|^{r(p)},
\ \forall \ x,y \in \R^d,
\end{equation}
\begin{equation*}
||\nabla v||_{L^{\infty}}\le C||v||_{W^{2,p}}\le C|| v||_{B_{p,q}^{2+\alpha}},
\end{equation*}
where $r(p)$ is given by
\begin{equation}\label{e7a}
r(p):=\begin{cases}
& \text{min}\{1+\alpha-\frac{d}{p},1\} ,\ \ \ \ \ \ \text{if}\ \ \ 1+\alpha-\frac{d}{p} \neq 1,\\
& \text{any} \ \text{real} \ \text{number} \  \in (\alpha,1),\ \ \  \text{if}\ \ \ 1+\alpha-\frac{d}{p} = 1.
\end{cases}
\end{equation}
Applying this to (\ref{e7aa}), we deduce that
\begin{equation*}
\begin{split}
&||\nabla^2F_v(\cdot +y)-\nabla^2F_v(\cdot )||_{L^p}\le C\Big(
||v||_{W^{2,p}}||\nabla^2 v(\cdot +y)-
\nabla^2 v(\cdot )||_{L^p}\\
&+|| v||_{W^{2,p}}||v(t)||_{B_{p,q}^{2+\alpha}}\big(|y|^{r(p)}
1_{\{|y|\le 1\}}+1_{\{|y|>1\}}\big)\Big),
\end{split}
\end{equation*}
which together with (\ref{e0}) and (\ref{e0a}), we conclude
that
\begin{equation*}
||F_v ||_{B^{2+\alpha}_{p,q}}\le C|| v||_{W^{2,p}}||v||_{B_{p,q}^{2+\alpha}}.
\end{equation*}
Combining the above estimates together, we may obtain the third
estimate in (\ref{e6}).

The second estimate in (\ref{e6}) may be proved following a similar procedure.
\end{proof}

\begin{rem}\label{r2.1}
According to potential theory, in general $F_v$ is not well defined without the $L^{p'}$ integrable condition on
$v$ for some $1<p'<\frac{d}{2}$. Note that (\ref{e6}) is independent of $p'$ and the $L^{p'}$ norm, although we assume
$v \in W^{2,p'}(\R^d;\R^d)$ to ensure that the Newton's potential $N$ is well defined.
By an approximation argument (see the proof of Proposition \ref{p1} below), if $\nabla \cdot v=0$, and
$v\in W^{2,p}(\R^d;\R^d)$ with some $p>d$, $F_v$ is still well defined.
\end{rem}

From now on, in this paper, for $1<p<\infty$, $1\le  q \le \infty$, $1<p'<\frac{d}{2}$, we define
\begin{equation}\label{e2a}
\begin{split}
&\S(p,q,p',T):=\Big\{ v \in C([0,T];C_b^{\infty}(\R^d;\R^d));\\
& \ \sup_{t \in [0,T]}
\big(||v(t)||_{B_{p,q}^{r}}+||v(t)||_{W^{2,p'}}\big)<\infty,\
\text{for}\
\forall\ r>1 ;
\ \ \
\nabla \cdot v(t)=0,\ \forall t \in [0,T] \Big\}.
\end{split}
\end{equation}

For every $v \in \S(p,q,p',T)$ with some $1<p<\infty$, $1\le q \le \infty$, $1<p'<\frac{d}{2}$,
$u_0\in C_b^{\infty}(\R^d;\R^d)$
with $\nabla \cdot u_0=0$,
$F_v$ is well defined and we can define
$\tilde \I_{\nu}(u_0,v)\in C([0,T];C_b^{\infty}(\R^d;\R^d))$ such that
$$\tilde \I_{\nu}(u_0,v)(t):=\mathbf{P}\big(Y_{T-t}^{T-t}(.)\big)$$
 for every
$t \in [0,T]$, where $\mathbf{P}$ denotes the Leray-Hodge projection on the space of
divergence free vector fields.
In particular, for every $v \in \S(p,q,p',T)$,
we define
$$\I_{\nu}(v):=\tilde \I_{\nu}(v(0),v).$$

Intuitively, if we can find a fixed point $v$  for the map
$\I_{\nu}$ (in some function space), then $v$ will be a solution of (\ref{e3}), hence a
solution of (\ref{e1}). In this work we will prove  that we can extend such map $\I_{\nu}$
to be defined in some Besov space, that $\I_{\nu}$ has a unique fixed point in such space, and that
the fixed point $v$
can be viewed as a solution of the Navier-Stokes equation (\ref{e1}).

We shall need the following result:
\begin{lem}\label{l3}
Suppose that  $v_m \in \S(p,q,p',T)$, $m=1,2$,
for some $d<p<\infty$, $1\le q \le \infty$, $1<p'<\frac{d}{2}$, $T>0$.
Then for every $0<\alpha<1$, $t \in [0,T]$, the functions $F_{v_m(t)}, G_{v_m(t)}$ defined by
(\ref{e5a})  satisfy the following inequalities,
\begin{equation}\label{e8aa}
\begin{split}
&||F_{v_1(t)}-F_{v_2(t)}||_{W^{1,p}}\le C_1
\sup_{m=1,2}||v_m(t)||_{W^{2,p}} ||v_1(t)-v_2(t)||_{W^{1,p}},\\
&||F_{v_1(t)}-F_{v_2(t)}||_{B^{1+\alpha}_{p,q}}\le C_1K||v_1(t)-v_2(t)||_{B^{1+\alpha}_{p,q}},
\end{split}
\end{equation}
where $K:=\sup_{t \in [0,T],\ m=1,2}||v_m(t)||_{B^{2+\alpha}_{p,q}}$, $C_1$ is independent of $v$, $K$, $p'$ and $T$.
\end{lem}
\begin{proof}
It is a consequence of Lemma 2.1,  the bilinearity
property of  the map $v\rightarrow F_v$ and the inequality,
\begin{equation}\label{e22a}
\begin{split}
&\Big|f_1(x)g_1(x)-f_2(x)g_2(x)-\big(
f_1(y)g_1(y)-f_2(y)g_2(y)\big)\Big|\\
&\le |g_1(x)|\big|f_1(x)-f_2(x)-(f_1(y)-f_2(y))\big|
+ |f_1(y)||\big|g_1(x)-g_2(x)-(g_1(y)-g_2(y))\big|\\
&+|f_2(x)-f_2(y)||g_1(x)-g_2(x)|
+|g_2(x)-g_2(y)||f_1(y)-f_2(y)|
\end{split}
\end{equation}
Indeed we can write
$F_{v_1(t)}-F_{v_2(t)}=\sum_i \nabla N \Big(\partial_i \big(f_{i,1}(t)-f_{i,2}(t)\big)\Big)$,
where $f_{i,m}(t,x):=\sum_j v_m^j(t,x)\partial_j v_m^i(t,x)$.
We have,
\begin{equation*}
\begin{split}
& ||F_{v_1(t)}-F_{v_2(t)}||_{L^p}\le C\sum_i ||f_{i,1}(t)-f_{i,2}(t)||_{L^p}\\
&\le C\Big(||\nabla v_1(t)||_{L^{\infty}}||v_1(t)-v_2(t)||_{L^p}+||v_2(t)||_{L^{\infty}}
||\nabla v_1(t)-\nabla v_2(t)||_{L^p}\Big)\\
&\le C\sup_{m=1,2}||v_m(t)||_{W^{2,p}} ||\nabla v_1(t)-\nabla v_2(t)||_{W^{1,p}},
\end{split}
\end{equation*}

As $\nabla \big(F_{v_1(t)}-F_{v_2(t)}\big)=\nabla^2 N \big(G_{v_1(t)}-G_{v_2(t)}\big)$,  we have,
\begin{equation*}
\begin{split}
&||\nabla F_{v_1(t)}-\nabla F_{v_2(t)}||_{L^p}\le C||G_{v_1(t)}-G_{v_2(t)}||_{L^p}\\
& \le C\big(||\nabla v_1(t)||_{L^{\infty}}+||\nabla v_2(t)||_{L^{\infty}}\big)||\nabla v_1(t)-\nabla v_2(t)||_{L^p}\\
&\le C\sup_{m=1,2}||v_m(t)||_{W^{2,p}}||\nabla v_1(t)-\nabla v_2(t)||_{W^{1,p}}.
\end{split}
\end{equation*}
therefore the first estimate holds.

Concerning the second estimate, writting  $\tilde G_{v_m(t)}^y(x):=G_{v_m(t)}(x+y)-G_{v_m(t)}(x)$, $m=1,2$,  by
(\ref{e22a}) we can get,
\begin{equation}\label{e23aa}
\begin{split}
& |\tilde G_{v_1(t)}^y(x)-\tilde G_{v_2(t)}^y(x)|\\
&\le C\Big(||\nabla v_1(t)||_{L^{\infty}}\big|\nabla v_1(t,x+y)-
\nabla v_1(t,x)-(\nabla v_2(t,x+y)-
\nabla v_2(t,x))\big|\\
&+\big|\nabla v_2(t,x+y)-\nabla v_2(t,x)\big|
\Big(|\nabla v_1(t,x+y)-\nabla v_2(t,x+y)|+
|\nabla v_1(t,x)-\nabla v_2(t,x)|\Big).
\end{split}
\end{equation}
Then according to (\ref{e7d}) we derive,
\begin{equation*}
\begin{split}
&||\big(\nabla F_{v_1(t)}(\cdot +y)-\nabla F_{v_1(t)}(\cdot )\big)-\big(\nabla F_{v_2(t)}(\cdot +y)-\nabla F_{v_2(t)}(\cdot )\big)||_{L^p}=
||\nabla^2 N \big(\tilde G_{v_1(t)}^y-\tilde G_{v_2(t)}^y\big)||_{L^p}\\
&\le C||\tilde G_{v_1(t)}^y-\tilde G_{v_2(t)}^y||_{L^p}
\le CK ||\big(\nabla v_1-\nabla v_2\big)(t,\cdot+y)-
\big(\nabla v_1-\nabla v_2\big)(t,\cdot)||_{L^p}\\
&+CK\big(|y|^{r(p)}1_{\{|y|\le 1\}}+1_{\{|y|> 1\}}\big)||\nabla v_1(t)- \nabla v_2(t)||_{L^p},
\end{split}
\end{equation*}
which implies that
\begin{equation*}
[\nabla F_{v_1(t)}-\nabla F_{v_2(t)}]_{B^{1+\alpha}_{p,q}}\le CK ||v_1(t)-v_2(t)||_{B^{1+\alpha}_{p,q}}.
\end{equation*}

\end{proof}



\section{The local existence theorem in  $B_{p,p}^{r}(\R^d;\R^d)$}

We first prove the following estimate.
\begin{lem}\label{l7}
Suppose that $v\in \S(p,p,p',T)$,
where $d<p<\infty$, $1<p'<\frac{d}{2}$, $T>0$.
Let $X$ be the
unique solution of the first equation of (\ref{e5}) (with coefficients $v$). Then for
every $0<\alpha<1$ and $f \in B_{p,p}^{2+\alpha}(\R^d)$, we have
\begin{equation}\label{e24a}
\sup_{0 \le s \le t \le T}||f\circ X_s^t||_{B_{p,p}^{2+\alpha}}\le C_1e^{C_1KT}
(1+T^{2}K^{2})||f||_{B_{p,p}^{2+\alpha}}, a.s.,
\end{equation}
where $f\circ X_s^t$ denotes the composition of function $f$ and the map $X_s^t:\R^d \rightarrow \R^d$,
$K:=\sup_{t \in [0,T]}||v(t)||_{B_{p,p}^{2+\alpha}}$, $C_1$ is a  positive constant independent of
$K$, $\nu$, $T$, $p'$, $v$ and $f$.
\end{lem}
\begin{proof}
{\bf Step 1}:
We need the $W^{2,p}$ estimate in \cite[Lemma 3.5]{Z} which is standard, but
for the reader's convenience, we follow the same method of that reference and
write the proof again.

Since $p>d$, by the Sobolev embedding theorem,
\begin{equation*}
\big(||v(t)||_{L^{\infty}}+||\nabla v(t)||_{L^{\infty}}\big)
\le C||v(t)||_{W^{2,p}}\le C||v(t)||_{B_{p,p}^{2+\alpha}} \le CK.
\end{equation*}
As $v \in \S(p,p,p',T)$, there is a version of
$X_s^t(.)$ which is $C^{\infty}$ differentiable and its first and second order differentials
$\nabla X_s^t$, $\nabla^2 X_s^t$ ($0\le t \le s \le T$) satisfy the following equation,
\begin{equation}\label{e7}
\begin{cases}
&d\nabla X_s^t(x)=-\nabla v(T-s,X_s^t(x))\nabla X_s^t(x)ds\\
&d\nabla^2 X_s^t(x)=-\nabla v(T-s,X_s^t(x))\nabla^2 X_s^t(x)ds-
\nabla^2 v(T-s,X_s^t(x))\big(\nabla X_s^t(x)\big)^2 ds\\
&\nabla X_s^t(x)=\mathbf{I},\ \ \nabla^2 X_s^t(x)=0,\ \
\end{cases}
\end{equation}
where $\mathbf{I}$ denotes identity map in $\R^d$.
Since $\nabla \cdot v(t)=0$, $X_s^t(.)$ and $(X_s^t)^{-1}(.)$ are volume preserving maps (see \cite{KU}),
 for every $h\in L^1(\R^d)$, $0\le t \le s \le T$, $h \circ X_s^t(\cdot)$ is
 almost surely a.e. well
 defined (with respect to Lebesgue measure), and
\begin{equation}\label{e10aa}
\int_{\R^d}h(X_s^t(x))dx=\int_{\R^d} h(x)dx,\ \ a.s..
\end{equation}
So we have, for $f\geqslant 0$,
\begin{equation}\label{l7.1}
\int_{\R^d}f^p(X_s^t(x))dx=\int_{\R^d} f^p(x)dx= ||f||_{L^p}^p,\ \ a.s..
\end{equation}

Note that $||\nabla v(t)||_{L^{\infty}}\le CK$ and the martingale part
of (\ref{e7}) vanishes. Applying Grownwall lemma, for every
$0\le t \le s \le T$, we have
\begin{equation}\label{e8}
|\nabla X_s^t(x)|\le Ce^{CKT},\ \ \ a.s, \ \ \forall\ x \in \R^d.
\end{equation}
Hence from (\ref{e7}),
\begin{equation*}
\begin{split}
&|\nabla^2 X_s^t(x)| \le K\int_t^s |\nabla^2 X_r^t(x)| dr+Ce^{CKT}\int_t^s
|\nabla^2 v(T-r,X_r^t(x))| dr,
\end{split}
\end{equation*}
then by Grownwall lemma, we derive that
\begin{equation*}
|\nabla^2 X_s^t(x)|\le Ce^{CKT}\int_t^s
|\nabla^2 v(T-r,X_r^t(x))| dr,
\end{equation*}
together with (\ref{e10aa}) and H\"older's inequality,
\begin{equation}\label{e9aa}
\begin{split}
&\int_{\R^d}|\nabla^2 X_s^t(x)|^p dx \le CT^{p-1}e^{CKT}\int_t^s \big(\int_{\R^d}|\nabla^2 v(T-r,X_r^t(x))|^p dx\big) dr\\
&\le CT^p e^{CKT}\sup_{t \in [0,T]}||v(t)||_{W^{2,p}}^p.
\end{split}
\end{equation}

Since $f \in B_{p,p}^{2+\alpha}(\R^d)$, $f \in C^1(\R^d)$ due to Sobolev embedding
theorem, so $\nabla (f \circ X_s^t)(x)=\nabla f (X_s^t(x))\nabla X_s^t(x)$, by (\ref{e10aa}) and
(\ref{e8}), we deduce that
\begin{equation}\label{l7.2}
\int_{\R^d}|\nabla (f \circ X_s^t)(x)|^p dx \le Ce^{CKT}||\nabla f||_{L^p}^p,\ \ a.s..
\end{equation}
Again note that if $f \in C^2(\R^d) \bigcap B_{p,p}^{2+\alpha}(\R^d)$,
\begin{equation}\label{e10a}
\nabla^2 (f \circ X_s^t)(x)=\nabla^2 f (X_s^t(x))(\nabla X_s^t(x))^2+
\nabla f (X_s^t(x))\nabla^2 X_s^t(x),
\end{equation}
by (\ref{e10aa})
(\ref{e8}) and (\ref{e9aa}),
\begin{equation}\label{l7.3}
\begin{split}
& \int_{\R^d} |\nabla^2 (f \circ X_s^t)(x)|^p dx\\
&\le C||\nabla X_s^t(.)||_{L^{\infty}}^{2p}
\int_{\R^d}|\nabla^2 f (X_s^t(x))|^p dx+C||\nabla f||_{L^{\infty}}^p
\int_{\R^d}|\nabla^2 X_s^t(x)|^p dx \\
&\le Ce^{CKT}||\nabla^2 f||_{L^p}^p+CT^pK^pe^{CKT}||\nabla f ||_{W^{1,p}}^p
\le C(1+T^pK^p)e^{CKT}||f||_{W^{2,p}}^p,
\end{split}
\end{equation}
where the second inequality follows from the Sobolev embedding theorem which in paricular implies that
$||\nabla f||_{L^{\infty}}\le C||\nabla f||_{W^{1,p}}$.

For general $f \in B_{p,p}^{2+\alpha}(\R^d)$, we choose a sequence
$\{f_n\}_{n=1}^{\infty}\subseteq C^2(\R^d)\bigcap B_{p,p}^{2+\alpha}(\R^d)$, such that
$\lim_{n \rightarrow \infty}||f_n-f||_{W^{2,p}}=0$, by (\ref{e10aa}), (\ref{e10a})
and approximation procedure, we know $f \circ X_s^t \in W^{2,p}(\R^d)$, and the estimate
(\ref{e10a}), (\ref{l7.3}) still hold.

{\bf Step 2}:  By (\ref{e8}), for every $0\le t \le s \le T$, $x,y \in \R^d$,
\begin{equation}\label{e10}
|X_s^t(x)-X_s^t(y)|\le Ce^{CKT}|x-y|,\ a.s..
\end{equation}
Let $\Gamma_s^t(x,y):=\nabla X_s^t(x)-\nabla X_s^t(y)$.
By (\ref{e7}) and noting that the martingale
part vanishes, we get that
\begin{equation}\label{e11}
\begin{split}
&|\Gamma_s^t(x,y)|
\le C\int_t^s\Big(\big(|\nabla v(T-r,X_r^t(x))||\Gamma_r^t(x,y)|\big)\\
&+
\big(|\nabla v(T-r,X_r^t(x))-\nabla v(T-r,X_r^t(y))||\nabla X_r^t(y)|\big)\Big)dr.
\end{split}
\end{equation}
Together with (\ref{e7d}), we deduce that
\begin{equation*}
|\nabla v(t,x)-\nabla v(t,y)|\le C ||v(t)||_{B_{p,p}^{2+\alpha}}|x-y|^{r(p)},
\ \forall \ x,y \in \R^d,
\end{equation*}
where $r(p)$ is defined by (\ref{e7a}). Hence for every
$x,y \in \R^d$,
\begin{equation}\label{e13}
|\nabla v(T-r,X_r^t(x))-\nabla v(T-r,X_r^t(y))|\le CK
|X_r^t(x)-X_r^t(y)|^{r(p)}\le CKe^{CKT}|x-y|^{r(p)}, a.s.,
\end{equation}
where we have used the estimate (\ref{e10}).
Applying such estimate to (\ref{e11}), we obtain
\begin{equation*}
\begin{split}
&|\Gamma_s^t(x,y)| \le CK\int_t^s
|\Gamma_r^t(x,y)| dr+CTKe^{CKT}|x-y|^{r(p)}.
\end{split}
\end{equation*}
Finally, by Grownwall lemma, for every $0\le t \le s \le T$,
\begin{equation}\label{e28}
|\nabla X_s^t(x)-\nabla X_s^t(y)|
=|\Gamma_s^t(x,y)| \le CT Ke^{CKT}|x-y|^{r(p)},\ a.s..
\end{equation}

{\bf Step 3}: Note that for every $h \in B_{p,p}^{\alpha}(\R^d)$,
\begin{equation}\label{e24c}
\begin{split}
&\int_{\R^d}\int_{\R^d}
\frac{|h (X_s^t(x+y))-h(X_s^t(x))|^p}{|y|^{d+\alpha p}}dx dy\\
&=\int_{\R^d}\int_{\R^d}
\frac{|h(X_s^t(y))-h(X_s^t(x))|^p}{|x-y|^{d+\alpha p}}dx dy\\
&=\int_{\R^d}\int_{\R^d}\frac{|h(y)-h(x)|^p}{|(X_s^t)^{-1}(x)-(X_s^t)^{-1}(y)|^{d+\alpha p}}dx dy,
\end{split}
\end{equation}
where in the last step we have  used (\ref{e10aa}) and the
change of variable. On the other hand, according to (\ref{e10}),
\begin{equation*}
\begin{split}
&|x-y|=|X_s^t\big((X_s^t)^{-1}(x)\big)-X_s^t\big((X_s^t)^{-1}(y)\big)|\le
Ce^{CKT}|(X_s^t)^{-1}(x)-(X_s^t)^{-1}(y)|,
\end{split}
\end{equation*}
which combining with (\ref{e24c}) yields that
\begin{equation}\label{e12}
\begin{split}
&\int_{\R^d}\int_{\R^d}
\frac{|h (X_s^t(x+y))-h(X_s^t(x))|^p}{|y|^{d+\alpha p}}dx dy\\
&\le Ce^{CKT}\int_{\R^d}\int_{\R^d}\frac{|h(y)-h(x)|^p}{|x-y|^{d+\alpha p}}dx dy
=Ce^{CKT}[h]_{B_{p,p}^{\alpha}}^p.
\end{split}
\end{equation}
{\bf Step 4}: Let $\Upsilon_s^t(x,y):=\nabla^2 X_s^t(x+y)- \nabla^2 X_s^t(x)$. By using similar arguments as
above, together with an application of  It\^o 's formula to
$\nabla^2 X_s^t(x+y)- \nabla^2 X_s^t(x)$, and using the estimates
(\ref{e8}), (\ref{e13}) and (\ref{e28}), we obtain
\begin{equation*}
\begin{split}
& |\Upsilon_s^t(x,y)| \le K \int_t^s
|\Upsilon_r^t(x,y)| dr+CKe^{CKT}\big(|y|^{r(p)}1_{\{|y|\le 1\}}
+1_{\{|y|>1\}}\big)\int_t^s |\nabla^2 X_r^t(x)| dr\\
&+Ce^{CKT}\int_t^s |\nabla^2 v(T-r, X_r^t(x+y))-\nabla^2 v(T-r, X_r^t(x))| dr\\
&+CT Ke^{CKT}\big(|y|^{r(p)}1_{\{|y|\le 1\}}
+1_{\{|y|>1\}}\big)\int_t^s  |\nabla^2 v(T-r, X_r^t(x))| dr, \ a.s.,
\end{split}
\end{equation*}
Applying Grownwall lemma,  H\"older inequality and  properties (\ref{e10aa}), (\ref{e12}) , we have,
\begin{equation}\label{e11a}
\begin{split}
&\int_{\R^d}\int_{\R^d}\frac{|\Upsilon_s^t(x,y)|^p}{|y|^{d+\alpha p}}dydx
\le Ce^{CKT}(T^pK^p+T^{2p}K^{2p}).
\end{split}
\end{equation}
{\bf Step 5}: By (\ref{e8}) and (\ref{e10a}), for every $0\le t \le s \le T$,
\begin{equation*}
\begin{split}
& [\nabla^2 (f\circ X_s^t)]_{B_{p,p}^{\alpha}}^p=
\int_{\R^d}\int_{\R^d}\frac{|\nabla^2 (f\circ X_s^t)(x+y)-
\nabla^2 (f\circ X_s^t)(x)|^p}{|y|^{d+\alpha p}}dydx\\
&\le C\int_{\R^d}\int_{\R^d}||\nabla f||_{L^{\infty}}^p
\frac{|\Upsilon_s^t(x,y)|^p}{|y|^{d+\alpha p}}dydx\\
&+C\big(\int_{\R^d}|\nabla^2  X_s^t(x)|^p dx\big)\big(
\int_{\R^d}\frac{||\nabla f \circ X_s^t(\cdot +y)-\nabla f \circ X_s^t(\cdot )||_{L^{\infty}}^p}{|y|^{d+\alpha p}}dy\big)\\
&+Ce^{CKT}\big(\int_{\R^d}|\nabla^2 f \circ X_s^t(x)|^p dx\big)
\big(\int_{\R^d}\frac{||\nabla X_s^t(\cdot +y)-\nabla X_s^t(\cdot )||_{L^{\infty}}^p}{|y|^{d+\alpha p}}dy\big)\\
&+Ce^{CKT}\int_{\R^d}\int_{\R^d}
\frac{|\nabla^2 f \circ X_s^t(x+y)-\nabla^2 f \circ X_s^t(x)|^p}{|y|^{d+\alpha p}}dydx\\
&:=\sum_{i=1}^4 I_i.
\end{split}
\end{equation*}
Due to (\ref{e11a}) and Sobolev embedding theorem, we have
$$I_1\le Ce^{CKT}(T^pK^p+T^{2p}K^{2p})||f||_{B_{p,p}^{2+\alpha}}^p.$$
By (\ref{e7d}) and (\ref{e10}), for every $y\in \R^d$,
\begin{equation*}
||\nabla f \circ X_s^t(\cdot +y)-\nabla f \circ X_s^t(\cdot )||_{L^{\infty}}
\le Ce^{CKT}||f||_{B_{p,p}^{2+\alpha}}\big(|y|^{r(p)}1_{\{|y|\le 1\}}+
1_{\{|y|> 1\}}\big);
\end{equation*}
combining this with (\ref{e9aa}), we get $I_2\le CT^pK^pe^{CKT}||f||_{B_{p,p}^{2+\alpha}}^p$.
According to (\ref{e8}) and (\ref{e28}), $I_3\le CT^pK^p e^{CKT}||f||_{B_{p,p}^{2+\alpha}}^p$.
By (\ref{e12}), $I_4\le Ce^{CKT}||f||_{B_{p,p}^{2+\alpha}}^p$. Putting these estimates together
we may conclude that
\begin{equation}\label{l7.4}
[\nabla^2 (f\circ X_s^t)]_{B_{p,p}^{\alpha}}^p\le Ce^{CKT}(1+
T^pK^p+T^{2p}K^{2p}))||f||_{B_{p,p}^{2+\alpha}}^p.
\end{equation}
Since $2TK\le 1+T^2K^2$, estimates (\ref{l7.1}), (\ref{l7.2}), (\ref{l7.3}) and (\ref{l7.4}) imply
(\ref{e24a}).
\end{proof}

\begin{rem}
If $p \neq q$, estimate (\ref{e24c}) is no longer useful, we can not get the corresponding estimate
(\ref{e24a}) in space $B_{p,q}^{2+\alpha}$ by the same method in Lemma \ref{l7}.
\end{rem}

Now we prove the $B_{p,p}^{2+\alpha}$ bounds for the regular solution of
(\ref{e5}).
\begin{lem}\label{l8}
Suppose that $v\in \S(p,p,p',T)$
where $d<p<\infty$, $1<p'<\frac{d}{2}$, $T>0$. Let
$(X,Y,Z)$ be the unique solution of (\ref{e5}) with coefficient $v$ and initial condition
$u_0:=v(0)$, and let $g(t,x):=Y_t^t(x)$. Then for any $0<\alpha<1$,
\begin{equation}\label{e24d}
\sup_{t \in [0,T]}||g(t)||_{B_{p,p}^{2+\alpha}} \le C_1e^{C_1KT}||u_0||_{B_{p,p}^{2+\alpha}}
\big(1+T^{2}K^{2}\big)+C_1TK^2(1+T^{2}K^{2}),
\end{equation}
where  $K:=\sup_{t \in [0,T]}||v(t)||_{B_{p,p}^{2+\alpha}}$, and $C_1$ is a constant independent of
$K$, $\nu$, $T$, $p'$ and $v$.
\end{lem}
\begin{proof}

Since $v \in \S(p,p,p',T)$, $F_v(t)\in C_b^{\infty}(\R^d;\R^d)$ by the
standard theorem on the regularity
 for the solution of elliptic equation (for example, see \cite{GT}) and
according to the computations in
\cite{PP}, we know that, for every $0\le t \le T$, $l>1$,
\begin{equation*}
\e\Big(\sup_{t\le s \le T }\big(|Y_s^t(x)|^l+|Z_s^t(x)|^l\big)\Big)<\infty.
\end{equation*}
Then taking the expectation in (\ref{e5}) and noticing that
$Y_t^t(x)$ is non-random (see \cite{PP}), we have,
\begin{equation}\label{e12a}
g(T-t,x)=Y_t^t(x)=\e\big(u_0(X_T^t(x))\big)+\int_t^T \e\big(F_v(T-s,X_s^t(x))\big)ds.
\end{equation}
Since
$F_v(t) \in B_{p,p}^{2+\alpha}(\R^d;\R^d)$ $\bigcap C_b^{\infty}(\R^d;\R^d)$ for every $t$, we can change the order of
expectation and differential, applying H\"older's  inequality, to obtain
\begin{equation*}
||\e\big(F_v(T-s,X_s^t(\cdot ))\big)||_{B_{p,p}^{2+\alpha}}^p\le
\e\Big(||F_v(T-s,X_s^t(\cdot ))||_{B_{p,p}^{2+\alpha}}^p\Big).
\end{equation*}
Hence
by Lemma \ref{l1} and \ref{l7}, for every $0\le t \le s \le T$,
\begin{equation*}
\begin{split}
&||\e\big(F_v(T-s,X_s^t(.))\big)||_{B_{p,p}^{2+\alpha}}\\
&\le Ce^{CKT}(1+T^{2} K^2)||F_v(T-s)||_{B_{p,p}^{2+\alpha}}\le
CKe^{CKT}(1+T^{2}K^2)||v(T-s)||_{B_{p,p}^{2+\alpha}}.
\end{split}
\end{equation*}
Similarly, for every $0\le t \le T$,
\begin{equation*}
||\e\big(u_0(X_T^t(\cdot ))\big)||_{B_{p,p}^{2+\alpha}}\le
Ce^{CKT}(1+T^{2}K^2)||u_0||_{B_{p,p}^{2+\alpha}}.
\end{equation*}
Putting the above estimate into (\ref{e12a}),   conclusion
(\ref{e24d}) follows.

\end{proof}

Consider vector fields $v_m \in \S(p,p,p',T)$,
where $d<p<\infty$, $1<p'<\frac{d}{2}$, $0<T<1$, $\ \ m=1,2$.
Frow now on in this section,
let $(X_m,Y_m,Z_m)$, where $m=1,2$, be the solutions of (\ref{e5}) with
coefficients $v_m$ and initial condition $u_{0,m}:=v_m(0)$. We will present some estimates on the difference between
$X_1$ and $X_2$.
\begin{lem}\label{l9}
For every $f_1,f_2\in W^{1,p}(\R^d)$ (recall that $p>d$), $0\le t\le s \le T$, we have,
\begin{equation}\label{e31}
\begin{split}
&\int_{\R^d}|f_1( X_{s,1}^t(x))-f_2(X_{s,2}^t(x))|^p dx\\
&\le C_1||f_1-f_2||_{L^p}^p+  C_1T^{p}e^{C_1KT}||\nabla f_2||_{L^p}^p\sup_{t \in [0,T]}||v_1(t)-v_2(t)||_{L^{\infty}}^p\ a.s.,
\end{split}
\end{equation}
where $K:=\sup_{t \in [0,T],\ m=1,2}||\nabla v_m(t)||_{L^{\infty}}$, and $C_1$
is a constant independent of $\nu$, $K$, $T$, $p'$, $f_m$ and $v_m$.
\end{lem}

\begin{proof}
We first assume $f_1,f_2 \in C^2(\R^d)\bigcap W^{1,p}(\R^d)$.
Since
\begin{equation*}
\begin{split}
&|f_1(X_{s,1}^t(x))-f_2( X_{s,2}^t(x))|\le
|f_1(X_{s,1}^t(x))-f_2( X_{s,1}^t(x))|+
|f_2(X_{s,1}^t(x))-f_2(X_{s,2}^t(x))|,
\end{split}
\end{equation*}
by (\ref{e10aa}), for every $0\le t \le s \le T$, we have
\begin{equation}\label{e16}
\int_{\R^d}|f_1(X_{s,1}^t(x))-f_2(X_{s,1}^t(x))|^p dx \le ||f_1-f_2||_{L^p}^p,\ a.s..
\end{equation}
For every $r\in [1,2]$, we define $X_{s}^{t,r}(x)$ to be the solution of following SDE,
\begin{equation}\label{e33}
\begin{cases}
&d X_{s}^{t,r}(x)=\sqrt{2\nu}dB_s-
\big((2-r)v_1(T-s,X_{s}^{t,r}(x))+(r-1)v_2(T-s,X_{s}^{t,r}(x))\big)ds,\\
&X_{t}^{t,r}(x)=x, \ \ 0\le t\le s\le T.
\end{cases}
\end{equation}
In particular $X_{s}^{t,r}(x)=X_{s,1}^t(x)$ when $r=1$ and
$X_{s}^{t,r}(x)=X_{s,2}^t(x)$ when $r=2$. Since
$\nabla \cdot \big((2-r)v_1+(r-1)v_2\big)=0$,  we have, for any
$h\in L^1(\R^d)$,
\begin{equation}\label{e31a}
\int_{\R^d} h(X_s^{t,r}(x))dx=\int_{\R^d}h(x)dx,\ \ \ \forall r\in [1,2], \ a.s..
\end{equation}
Since $v_m$ is regular enough,  the methods of \cite{KU} allow us to check that
there is a version of
$X_s^{t,r}(x)$ which is differentiable with $r$ and that $V_s^{t,r}(x):=\frac{d}{dr}X_s^{t,r}(x)$
satisfies the following SDE,
\begin{equation}\label{e33a}
\begin{cases}
&dV_s^{t,r}(x)=-\big((2-r)\nabla v_1(T-s,X_{s}^{t,r}(x))+(r-1)\nabla v_2(T-s,X_{s}^{t,r}(x))\big)V_s^{t,r}(x)ds\\
&+\big(v_1(T-s,X_{s}^{t,r}(x))-v_2(T-s,X_{s}^{t,r}(x))\big)ds,\\
&V_t^{t,r}(x)=0,\ \ \  0\le t\le s\le T.
\end{cases}
\end{equation}
By Grownwall lemma for every
$0\le t\le s \le T$, $r\in [1,2]$ and $x\in \R^d$,
\begin{equation}\label{e32}
|V_s^{t,r}(x)|\le
CTe^{CKT}\sup_{t \in [0,T]}||v_1(t)-v_2(t)||_{L^{\infty}} ,\ a.s..
\end{equation}
Since
\begin{equation*}
\begin{split}
|f_2(X_{s,1}^t(x))-f_2(X_{s,2}^t(x))|^p=|\int_1^2
\frac{d}{dr}\big(f_2(X_{s}^{t,r}(x))\big)dr|^p
\le \int_1^2 |\nabla f_2(X_{s}^{t,r}(x))|^p|V_s^{t,r}(x)|^pdr,
\end{split}
\end{equation*}
then, by (\ref{e31a}) and (\ref{e32}), we have,
\begin{equation}\label{e16a}
\begin{split}
& \int_{\R^d}|f_2(X_{s,1}^t(x))-f_2(X_{s,2}^t(x))|^p dx\\
&\le CT^p e^{CKT}\sup_{t \in [0,T]}||v_1(t)-v_2(t)||_{L^{\infty}}^p\int_1^2
\int_{\R^d} |\nabla f_2(X_{s}^{t,r}(x))|^p dx dr\\
&\le CT^p e^{CKT}\int_{\R^d}|\nabla f_2(x)|^p dx \sup_{t \in [0,T]}||v_1(t)-v_2(t)||_{L^{\infty}}^p, \ a.s..
\end{split}
\end{equation}
Combining (\ref{e16}) and (\ref{e16a}) together, we have (\ref{e31}).

For general $f_1,f_2 \in W^{1,p}(\R^d)$, there exist sequences
$\{f_{1,n}\}_{n=1}^{\infty}$, $\{f_{2,n}\}_{n=1}^{\infty}$
$\subseteq C^2(\R^d)\bigcap W^{1,p}(\R^d)$, such that
\begin{equation*}
\lim_{n \rightarrow \infty}\sup_{x \in \R^d}
|f_{i,n}(x)-f_i(x)|\le
\lim_{n \rightarrow \infty}||f_{i,n}-f_i||_{W^{1,p}}=0,\ \
\sup_n ||f_{i,n}||_{W^{1,p}}\le ||f_i||_{W^{1,p}},\ i=1,2,
\end{equation*}
then according to Fatou lemma,
\begin{equation*}
\begin{split}
&\int_{\R^d}|f_1( X_{s,1}^t(x))-f_2(X_{s,2}^t(x))|^p dx
=\int_{\R^d}\lim_{n \rightarrow \infty}|f_{1,n}( X_{s,1}^t(x))-f_{2,n}(X_{s,2}^t(x))|^p dx\\
&\le \liminf_{n \rightarrow \infty} \int_{\R^d}|f_{1,n}( X_{s,1}^t(x))-f_{2,n}(X_{s,2}^t(x))|^p dx\\
&\le C||f_1-f_2||_{L^p}^p+  CT^{p}e^{C_1KT}||\nabla f_2||_{L^p}^p\sup_{t \in [0,T]}||v_1(t)-v_2(t)||_{L^{\infty}}^p\ a.s..
\end{split}
\end{equation*}
\end{proof}

\begin{lem}\label{l10}
Suppose that $f_1, f_2 \in B_{p,p}^{1+\alpha}(\R^d)$ where $d<p<\infty$,
$0<\alpha<1$. We have,
for every $0\le t \le s \le T$,
\begin{equation}\label{e32a}
\begin{split}
& \int_{\R^d}\int_{\R^d}\frac{\Big|\big(f_1(X_{s,1}^t(x+y))-f_2(X_{s,2}^t(x+y))\big)-
\big(f_1(X_{s,1}^t(x))-f_2(X_{s,2}^t(x))\big)\Big|^p}{|y|^{3+p\alpha}} dxdy \\
& \le C_1e^{C_1KT}[f_1-f_2]_{B_{p,p}^{\alpha}}^p
+C_1T^{p}(1+K^{p})e^{C_1KT} ||f_2||_{B_{p,p}^{1+\alpha}}^p
\sup_{t \in [0,T]}||v_1(t)-v_2(t)||_{B_{p,p}^{1+\alpha}}^p\ a.s.,
\end{split}
\end{equation}
where $K:=\sup_{t \in [0,T],\ m=1,2}||v_m(t)||_{B_{p,p}^{2+\alpha}}$, and $C_1$
is a positive constant independent of $\nu$, $K$, $T$, $p'$, $f_m$, and $v_m$.
\end{lem}
\begin{proof}
Since $C_c^{\infty}(\R^d)$ is dense in $B_{p,p}^{1+\alpha}(\R^d)$ (see \cite[Theorem 2.3.2(a)]{T}),
by the same approximation argument in the proof of Lemma
\ref{l9}, it suffices to show (\ref{e32a}) for every $f_1, f_2 \in C^2(\R^d)\bigcap B_{p,p}^{1+\alpha}(\R^d)$.

We have,
\begin{equation*}
\begin{split}
& \Big|\big(f_1(X_{s,1}^t(x+y))-f_2(X_{s,2}^t(x+y))\big)-
\big(f_1(X_{s,1}^t(x))-f_2(X_{s,2}^t(x))\big)\Big|\\
&\le  \Big|\big(f_1(X_{s,1}^t(x+y))-f_1(X_{s,1}^t(x))\big)-
\big(f_2(X_{s,1}^t(x+y))-f_2(X_{s,1}^t(x))\big)\Big|\\
&+
 \Big|\big(f_2(X_{s,1}^t(x+y))-f_2(X_{s,1}^t(x))\big)-
\big(f_2(X_{s,2}^t(x+y))-f_2(X_{s,2}^t(x))\big)\Big|\\
&:=I_{s,1}^t(x,y)+I_{s,2}^t(x,y).
\end{split}
\end{equation*}
Note that
\begin{equation*}
I_{s,1}^t(x,y)=\big(f_1-f_2\big)(X_{s,1}^t(x+y))-
\big(f_1-f_2\big)(X_{s,1}^t(x)),
\end{equation*}
by  property (\ref{e12}) in the proof of Lemma \ref{l7}, we have,
\begin{equation*}
\int_{\R^d}\int_{\R^d}\frac{|I_{s,1}^t(x,y)|^p}{|y|^{3+p\alpha}}dxdy \le Ce^{CKT}
||f_1-f_2||_{B_{p,p}^{\alpha}}, a.s..
\end{equation*}
For $r\in [0,1]$, let $X_s^{t,r}(x)$, $V_s^{t,r}(x)$ be the solutions of
SDEs (\ref{e33}) and (\ref{e33a}) respectively. Hence
\begin{equation}\label{e34}
\begin{split}
&|I_{s,2}^t(x,y)|^p=\Big|\int_1^2 \frac{d}{dr}
\big(f_2(X_{s}^{t,r}(x+y))-f_2(X_{s}^{t,r}(x))\big)dr\Big|^p\\
&\le C\int_1^2 \big|\nabla f_2(X_s^{t,r}(x+y))\big|^p
\big|V_s^{t,r}(x+y)-V_s^{t,r}(x)\big|^p dr\\
&+C\int_1^2\big|\nabla f_2(X_s^{t,r}(x+y))-\nabla f_2(X_s^{t,r}(x))\big|^p
\big|V_s^{t,r}(x)\big|^p dr.
\end{split}
\end{equation}
Let $\Gamma_s^{t,r}(x,y):=V_s^{t,r}(x+y)-V_s^{t,r}(x)$. Applying It\^o 's formula in
(\ref{e33a}), we derive that
\begin{equation}\label{e14a}
\begin{split}
&|\Gamma_s^{t,r}(x,y)|\le K \int_s^T |\Gamma_u^{t,r}(x,y)|du\\
&+\int_s^T \big|\nabla v_r(T-u,X_u^{t,r}(x+y))-\nabla v_r(T-u,X_u^{t,r}(x))\big|
\big|V_u^{t,r}(x)\big|du\\
&+\int_s^T \big|(v_2-v_1)(T-u,X_u^{t,r}(x+y))-
(v_2-v_1)(T-u,X_u^{t,r}(x))\big|du
\end{split}
\end{equation}
where
$v_r(s,x):=(2-r)\nabla v_1(s,x)+(r-1)\nabla v_2(s,x)$.

In the same way we proved (\ref{e8}) and (\ref{e10}), for every $r\in[1,2]$, $0\le t\le s\le T$,
$x,y \in \R^d$,
\begin{equation}\label{e34a}
|X_{s}^{t,r}(x+y)-X_{s}^{t,r}(x)|\le Ce^{CKT}|y|,\ \ a.s..
\end{equation}
Hence according to (\ref{e7d}) and (\ref{e34a}),
\begin{equation*}
\begin{split}
&\big|\nabla v_r(T-u,X_u^{t,r}(x+y))-\nabla v_r(T-u,X_u^{t,r}(x))\big|\\
&\le Ce^{CKT}||v_r(T-u)||_{B_{p,p}^{2+\alpha}}\big(|y|^{r(p)}
1_{\{|y|\le 1\}}+ 1_{\{|y|>1\}}\big)\\
&\le CKe^{CKT}\big(|y|^{r(p)}
1_{\{|y|\le 1\}}+ 1_{\{|y|>1\}}\big),\ a.s..
\end{split}
\end{equation*}
As in (\ref{e7d}), using the embedding theorem \cite[Theorem 2.8.1]{T} and (\ref{e34a}), we can show
that for every $r\in [1,2]$, $0\le t \le u \le T$,
\begin{equation*}
\begin{split}
&\big|(v_2-v_1)(T-u,X_u^{t,r}(x+y))-
(v_2-v_1)(T-u,X_u^{t,r}(x))\big|\\
&\le Ce^{CKT}||v_1(T-u)- v_2(T-u)||_{B_{p,p}^{1+\alpha}}
\big(|y|^{r(p)}1_{\{|y|\le 1\}}
+1_{\{|y|> 1\}}\big),\ a.s..
\end{split}
\end{equation*}
Combining the above estimate and (\ref{e32}) together into (\ref{e14a}), we have,
\begin{equation}\label{e13a}
\begin{split}
&|\Gamma_s^{t,r}(x,y)|
\le CT(1+KT)e^{CKT}\sup_{t \in [0,T]}|| v_1(t)-v_2(t)||_{B_{p,p}^{1+\alpha}}
\big(|y|^{r(p)}
1_{\{|y|\le 1\}}+ 1_{\{|y|>1\}}\big),\ a.s..
\end{split}
\end{equation}
Analogously to the proof of (\ref{e12}) in  Lemma \ref{l7}, we can show that
for every $r\in [1,2]$,
\begin{equation}\label{e14}
\begin{split}
&\int_{\R^d}\int_{\R^d}\frac{|\nabla f_2(X_s^{t,r}(x+y))-
\nabla f_2(X_s^{t,r}(x))|^p}{|y|^{d+\alpha p}}dxdy
\le Ce^{CKT}||f_2||_{B_{p,p}^{1+\alpha}}^p,\ a.s..
\end{split}
\end{equation}
Combining (\ref{e13a}) and (\ref{e14})  into (\ref{e34}) and using  properties (\ref{e31a}),
(\ref{e32}), we obtain
\begin{equation*}
\begin{split}
&\int_{\R^d}\int_{\R^d}\frac{|I_{s,2}^t(x,y)|^p}{|y|^{3+p\alpha}}dxdy\\
&\le CT^p(1+K^pT^p)e^{CKT} ||f_2||_{B_{p,p}^{1+\alpha}}^p
\sup_{t \in [0,T]}||v_1(t)-v_2(t)||_{B_{p,p}^{1+\alpha}}^p.
\end{split}
\end{equation*}
Since we assume that $0<T<1$, by putting the estimates for $I_{s,1}^t$ and $I_{s,2}^t$ together, we can prove
(\ref{e32a}).
\end{proof}

Based on Lemma \ref{l9} and \ref{l10}, we can prove the following difference estimate.
\begin{lem}\label{l11}
Suppose that  $f_1, f_2 \in B_{p,p}^{2+\alpha}(\R^d)$ for some $d<p<\infty$,
$0<\alpha<1$. Then for
every $0\le t \le s\le T$,
\begin{equation}\label{e15}
\begin{split}
&||f_1\circ X_{s,1}^t(\cdot )-f_2\circ X_{s,2}^t(\cdot )||_{B_{p,p}^{1+\alpha}}
\le C_1e^{C_1KT}(1+TK)||f_1-f_2||_{B_{p,p}^{1+\alpha}}\\
&+C_1T e^{C_1KT}\big(1+T^2K^{2}\big)||f_2||_{B_{p,p}^{2+\alpha}}\sup_{t \in [0,T]}
||v_1(t)-v_2(t)||_{B_{p,p}^{1+\alpha}}\ a.s,
\end{split}
\end{equation}
where $K:=\sup_{t \in [0,T],\ m=1,2}||v_m(t)||_{B_{p,p}^{2+\alpha}}$, and $C_1$
is a positive constant independent of $\nu$, $K$, $T$, $p'$, $f_m$ and $v_m$.
\end{lem}
\begin{proof}
{\bf Step 1}:  By Lemma \ref{l9} and applying Sobolev embedding theorem,
for every $0\le t \le s \le T$,
\begin{equation}\label{l11.1}
\begin{split}
&\int_{\R^d}|f_1(X_{s,1}^t(x))-f_2( X_{s,2}^t(x))|^p dx\\
&\le C||f_1-f_2||_{L^p}^p+  CT^p e^{CKT}||f_2||_{W^{1,p}}^p\sup_{t \in [0,T]}||v_1(t)-v_2(t)||_{W^{1,p}}^p,\ a.s..
\end{split}
\end{equation}
Note that for $m=1,2$, $\nabla (f_m \circ X_{s,m}^t)(x)$
$=\nabla f_m (X_{s,m}^t(x)) \nabla X_{s,m}^t(x)$, we have,
\begin{equation}\label{e15a}
\begin{split}
& |\nabla (f_1 \circ X_{s,1}^t)(x)-\nabla (f_2 \circ X_{s,2}^t)(x)|\\
&\le Ce^{CKT}|\nabla f_1 ( X_{s,1}^t(x))-\nabla f_2 (X_{s,2}^t(x))|
+C||\nabla f_2||_{L^{\infty}}|\nabla X_{s,1}^t(x)-\nabla X_{s,2}^t(x)|,
\end{split}
\end{equation}
where we use the estimate (\ref{e8}).
Applying (\ref{l11.1}), we have,
\begin{equation*}
\begin{split}
&\int_{\R^d}|\nabla f_1(X_{s,1}^t(x))-\nabla f_2( X_{s,2}^t(x))|^p dx\\
&\le C||f_1-f_2||_{W^{1,p}}^p+ CT^p e^{CKT}||f_2||_{W^{2,p}}^p\sup_{t \in [0,T]}||v_1(t)-v_2(t)||_{W^{1,p}}^p,
\ a.s..
\end{split}
\end{equation*}
Let $\Gamma_s^t(x):=\nabla X_{s,1}^t(x)-\nabla X_{s,2}^t(x)$. Applying It\^o 's formula in the first equation in
(\ref{e7}) and using the estimate (\ref{e8}), we get that
\begin{equation*}
\begin{split}
& |\Gamma_s^t(x)| \le CK \int_t^s |\Gamma_r^t(x)| dr+
Ce^{CKT}\int_t^s |\nabla v_1(T-r,X_{r,1}^t(x))-\nabla v_2(T-r,X_{r,2}^t(x))| dr,
\end{split}
\end{equation*}
hence,  by applying (\ref{l11.1}) to $f_1=\nabla v_1$, $f_2=\nabla v_2$ and using Grownwall lemma, together with H\"older's inequality, we may obtain
\begin{equation}\label{e17a}
\int_{\R^d}|\Gamma_s^t(x)|^p dx\le CT^pe^{CKT}(1+T^pK^p)\sup_{t \in [0,T]}||v_1(t)-v_2(t)||_{W^{1,p}}^p\ a.s..
\end{equation}
Putting the above estimates into (\ref{e15a}) and noticing that $0<T<1$, we have,
\begin{equation}\label{l11.2}
\begin{split}
&\int_{\R^d}|\nabla (f_1 \circ X_{s,1}^t)(x)-\nabla (f_2 \circ X_{s,2}^t)(x)|^p dx\\
&\le Ce^{CKT}||f_1-f_2||_{W^{1,p}}^p+
CT^pe^{CKT}(1+T^pK^p)||f_2||_{W^{2,p}}^p \sup_{t \in [0,T]}||v_1(t)-v_2(t)||_{W^{1,p}}^p,\ a.s..
\end{split}
\end{equation}

{\bf Step 2}: By (\ref{e22a}), for every $0\le t \le s \le T$, $x,y \in \R^d$,
\begin{equation}\label{e17}
\begin{split}
&\Big|\big(\nabla (f_1\circ X_{s,1}^t)(x+y)-
\nabla (f_2\circ X_{s,2}^t)(x+y)\big)-
\big(\nabla (f_1\circ X_{s,1}^t)(x)-
\nabla (f_2\circ X_{s,2}^t)(x)\big)\Big|\\
&\le Ce^{CKT}\Big|
\big(\nabla f_1(X_{s,1}^t(x+y))-\nabla f_2(X_{s,2}^t(x+y))\big)-
\big(\nabla f_1(X_{s,1}^t(x))- \nabla f_2(X_{s,2}^t(x))\big)\Big|\\
&+C||\nabla f_1||_{L^{\infty}}\big|\big(\nabla X_{s,1}^t(x+y)-\nabla X_{s,2}^t(x+y)\big)
-\big(\nabla X_{s,1}^t(x)-\nabla X_{s,2}^t(x)\big)\big|\\
&+C|\nabla X_{s,1}^t(x+y)-\nabla X_{s,2}^t(x+y)|\big|\nabla f_2(X_{s,2}^t(x+y))-\nabla f_2(X_{s,2}^t(x))\big|\\
&+C|\nabla X_{s,2}^t(x+y)-\nabla X_{s,2}^t(x)|\big|\nabla f_1(X_{s,1}^t(x))-\nabla f_2(X_{s,2}^t(x))\big|\\
&:=\sum_{i=1}^4 I_{s,i}^t(x,y),
\end{split}
\end{equation}
where we have used estimate (\ref{e8}).
According to Lemma \ref{l10},
\begin{equation*}
\begin{split}
&\int_{\R^d}\int_{\R^d}\frac{|I_{s,1}^t(x,y)|^p}{|y|^{d+\alpha p}}dydx\\
&\le Ce^{CKT}||f_1-f_2||_{B_{p,p}^{1+\alpha}}+CT^p e^{CKT}(1+ K^p)
||f_2||_{B_{p,p}^{2+\alpha}}^p
\sup_{t \in [0,T]}||v_1(t)-v_2(t)||_{B_{p,p}^{1+\alpha}}^p\ a.s..
\end{split}
\end{equation*}
From (\ref{e7d}) and (\ref{e10}),
\begin{equation*}
\begin{split}
&\big|\nabla f_2(X_{s,2}^t(x+y))-\nabla f_2(X_{s,2}^t(x))\big|\\
&\le C||f_2||_{B_{p,p}^{2+\alpha}}\Big(
|X_{s,2}^t(x+y)-X_{s,2}^t(x)|^{r(p)}1_{\{|y|\le 1\}}+1_{\{|y|>1\}}\Big)\\
&\le Ce^{CKT}||f_2||_{B_{p,p}^{2+\alpha}}
\big(|y|^{r(p)}1_{\{|y|\le 1\}}+1_{\{|y|>1\}}\big),\  a.s.
\end{split}
\end{equation*}
Combining this with (\ref{e17a}), we deduce that
\begin{equation*}
\begin{split}
&\int_{\R^d}\int_{\R^d}\frac{|I_{s,3}^t(x,y)|^p}{|y|^{d+\alpha p}}dydx\\
&\le CT^p e^{CKT}(1+T^pK^p)||f_2||_{B_{p,p}^{2+\alpha}}^p
\sup_{t \in [0,T]}||v_1(t)-v_2(t)||_{W^{1,p}}^p\ a.s..
\end{split}
\end{equation*}
According to estimate (\ref{e28}) and Lemma \ref{l9},
\begin{equation*}
\begin{split}
&\int_{\R^d}\int_{\R^d}\frac{|I_{s,4}^t(x,y)|^p}{|y|^{d+\alpha p}}dydx\\
&\le CT^pK^pe^{CKT}\Big(||f_1-f_2||_{W^{1,p}}^p+
T^p||f_2||_{W^{2,p}}^p\sup_{t \in [0,T]}||v_1(t)-v_2(t)||_{W^{1,p}}^p\Big)\ a.s..
\end{split}
\end{equation*}
Let $\Psi_s^t(x,y):=(\nabla X_{s,1}^t(x+y)- \nabla X_{s,2}^t(x+y))-(\nabla X_{s,1}^t(x)-\nabla X_{s,2}^t(x))$. Applying
It\^o's formula in (\ref{e7}),
and estimate (\ref{e8}), we have
\begin{equation*}
\begin{split}
&|\Psi_s^t(x,y)| \le CK \int_t^s |\Psi_r^t(x,y)| dr\\
&+Ce^{CKT}\int_t^s \big|\big(\nabla v_1(T-r,X_{r,1}^t(x+y))-
\nabla v_2(T-r,X_{r,2}^t(x+y))\big)\\
&-\big(\nabla v_1(T-r,X_{r,1}^t(x))-
\nabla v_2(T-r,X_{r,2}^t(x))\big)\big| dr\\
&+ C\int_t^s|\nabla X_{r,1}^t(x+y)-\nabla X_{r,2}^t(x+y)|
\big|\nabla v_2(T-r,X_{r,2}^t(x+y))-\nabla v_2(T-r,X_{r,2}^t(x))\big| dr\\
&+C\int_t^s|\nabla X_{r,2}^t(x+y)-\nabla X_{r,2}^t(x)|
\big|\nabla v_1(T-r,X_{r,1}^t(x))-\nabla v_2(T-r,X_{r,2}^t(x))\big| dr.
\end{split}
\end{equation*}
As in the previous argument
first applying Grownwall lemma, then using H\"older inequality and
estimating the drift terms in the same way as for
$I_{s,i}^t(x,y),\ i=1,3,4$,  we may obtain
\begin{equation*}
\begin{split}
&\int_{\R^d}\int_{\R^d}\frac{|\Psi_s^t(x,y)|^p}{|y|^{d+\alpha p}}dydx\\
&\le CT^{p}e^{CKT}\big(1+T^{2p}K^{2p}\big)\sup_{t \in [0,T]}||v_1(t)-v_2(t)||_{B_{p,p}^{1+\alpha}}^p\ a.s,
\end{split}
\end{equation*}
where we have used the assumption that $0<T<1$.
Plugging the estimates obtained above together into (\ref{e17}), we obtain for every
$0\le t \le s \le T$ that
\begin{equation}\label{l11.3}
\begin{split}
&[f_1 \circ X_{s,1}^t-f_2 \circ X_{s,2}^t]_{B_{p,p}^{1+\alpha}}^p
\le Ce^{CKT}(1+T^pK^p)||f_1-f_2||_{B_{p,p}^{1+\alpha}}^p\\
&+CT^p e^{CKT}\big(1+T^{2p}K^{2p})\big)||f_2||_{B_{p,p}^{2+\alpha}}^p\sup_{t \in [0,T]}
||v_1(t)-v_2(t)||_{B_{p,p}^{1+\alpha}}^p\ a.s..
\end{split}
\end{equation}
Putting (\ref{l11.1}), (\ref{l11.2}) and
(\ref{l11.3}) together to conclude the proof.
\end{proof}


We are now in a position to prove the following difference estimate which will be used to prove the local existence.
\begin{lem}\label{l12}
Let $g_m(t,x):=Y_{t,m}^t(x)$, $m=1,2$, where
$d<p<\infty$, $0<\alpha<1$, $0\le t \le T$. Then
\begin{equation}\label{e35}
\begin{split}
&\sup_{t \in [0,T]}||g_1(t)-g_2(t)||_{B_{p,p}^{1+\alpha}}
\le  C_1e^{C_1KT}(1+T K)||u_{0,1}-u_{0,2}||_{B_{p,p}^{1+\alpha}}\\
&+CTK(1+T^3K^3)e^{CKT}\sup_{t \in [0,T]}||v_1(t)-v_2(t)||_{B_{p,p}^{1+\alpha}},
\end{split}
\end{equation}
where $K:=\sup_{t \in [0,T], m=1,2}||v_m(t)||_{B_{p,p}^{2+\alpha}}$, and $C_1$ is a
positive constant independent of
$K$, $\nu$, $T$, $p'$ and $v_m$.
\end{lem}
\begin{proof}
As (\ref{e12a}), for $m=1,2$, $0\le t \le T$,
\begin{equation}\label{e18a}
g_m(t)=Y_{t,m}^t(x)=\e\big(u_{0,m}(X_{T,m}^t(x))\big)+\int_t^T \e\big(F_{v_m}(T-s,X_{s,m}^t(x))\big)ds.
\end{equation}

According to the regularity theorem of elliptic equation,
$F_v(t) \in C_b^{\infty}(\R^d;\R^d)$ for every $t$, so that we can change the order of
expectation and differential, together with H\"older's  inequality, to obtain
\begin{equation}\label{e18}
\begin{split}
&||\e\big(F_{v_1}(T-s,X_{s,1}^t(\cdot ))\big)-\e\big(F_{v_2}(T-s,X_{s,2}^t(\cdot ))\big)||_{B_{p,p}^{1+\alpha}}^p\\
&\le\e\Big(||F_{v_1}(T-s,X_{s,1}^t(\cdot ))-F_{v_2}(T-s,X_{s,2}^t(\cdot ))||_{B_{p,p}^{1+\alpha}}^p\Big).
\end{split}
\end{equation}
By Lemma \ref{l11}, for every $0\le t \le s \le T$,
\begin{equation*}
\begin{split}
&||F_{v_1}(T-s,X_{s,1}^t(\cdot ))-F_{v_2}(T-s,X_{s,2}^t(\cdot ))||_{B_{p,p}^{1+\alpha}}^p\\
&\le Ce^{CKT}(1+T^pK^p)||F_{v_1}(T-s)-F_{v_2}(T-s)||_{B_{p,p}^{1+\alpha}}^p\\
&+ CT^p e^{CKT}(1+T^{2p}K^{2p})
\sup_{t \in [0,T]}||F_{v_2}(T-s)||_{B_{p,p}^{2+\alpha}}^p||v_1(t)-v_2(t)||_{B_{p,p}^{1+\alpha}}^p\ a.s.,
\end{split}
\end{equation*}
thus, according to Lemmas \ref{l1}, \ref{l3},
\begin{equation*}
\begin{split}
&||F_{v_1}(T-s,X_{s,1}^t(\cdot ))-F_{v_2}(T-s,X_{s,2}^t(\cdot ))||_{B_{p,p}^{1+\alpha}}^p\\
&\le CK^p e^{CKT}\big(1+T^{3p}K^{3p}  \big)\sup_{t \in [0,T]}||v_1(t)-v_2(t)||_{B_{p,p}^{1+\alpha}}^p\ a.s.
\end{split}
\end{equation*}
Putting this into (\ref{e18}),
\begin{equation*}
\begin{split}
&||\e\big(F_{v_1}(T-s,X_{s,1}^t(\cdot ))\big)-\e\big(F_{v_2}(T-s,X_{s,2}^t(\cdot ))\big)||_{B_{p,p}^{1+\alpha}}\\
&\le CKe^{CKT}\big(1+T^3K^3\big)\sup_{t \in [0,T]}||v_1(t)-v_2(t)||_{B_{p,p}^{1+\alpha}}.
\end{split}
\end{equation*}
Similarly, we can show that,
\begin{equation*}
\begin{split}
&||\e\big(u_{0,1}(X_{T,1}^t(\cdot ))\big)-\e\big(u_{0,2}(X_{T,2}^t(\cdot ))\big)||_{B_{p,p}^{1+\alpha}}\\
&\le Ce^{CKT}(1+T K)||u_{0,1}-u_{0,2}||_{B_{p,p}^{1+\alpha}}
+CT K(1+T^2 K^2)e^{CKT}\sup_{t \in [0,T]}||v_1(t)-v_2(t)||_{B_{p,p}^{1+\alpha}}.
\end{split}
\end{equation*}
Putting the above estimate  into (\ref{e18a}), we have (\ref{e35}).
\end{proof}

\begin{rem}\label{r1}
Note that in the estimate (\ref{e35}) for $||g_1(t)-g_2(t)||_{B_{p,p}^{1+\alpha}}$,
the difference term $||v_1-v_2||_{B_{p,p}^{1+\alpha}}$ is considered with the
$B_{p,p}^{1+\alpha}$ norm, and the uniformly control term is  with the  $B_{p,p}^{2+\alpha}$ norm (see
the definition of $K$), which is one order higher.
But in the estimate (\ref{e24d}) for $||g(t)||_{B_{p,p}^{2+\alpha}}$, only $B_{p,p}^{2+\alpha}$ norm is
involved, which is of the same order as the one of $g(t)$.

Since for $p>1$ and $r>1+\frac{d}{p}$, $||\nabla v(t)||_{L^{\infty}}
\le C||v(t)||_{B_{p,q}^r}$, repeating the procedure used above, we can also
obtain the following estimate with lower and higher order Besov norm.

For every $p>1$ and $r>1+\frac{d}{p}$,
\begin{equation}\label{e19aa}
\begin{split}
&\sup_{t \in [0,T]}||g_1(t)-g_2(t)||_{B_{p,p}^{\max(r-1,1)}}
\le  C_1e^{C_1KT}(1+TK^{[r]-1})||u_{0,1}-u_{0,2}||_{B_{p,p}^{\max(r-1,1)}}\\
&+CTK(1+T^{[r]+1}K^{[r]+1})e^{CKT}\sup_{t \in [0,T]}||v_1(t)-v_2(t)||_{B_{p,p}^{\max(r-1,1)}}
\end{split}
\end{equation}
and for $m=1,2$,
\begin{equation}\label{e19}
\sup_{t \in [0,T]}||g_m(t)||_{B_{p,p}^{r}} \le C_1(1+T^{[r]}K^{[r]})e^{C_1KT}||u_0||_{B_{p,p}^r}
+C_1TK^2(1+T^{[r]+1}K^{[r]+1}),
\end{equation}
where $r:=[r]+[r]^+$ with $[r]$ to be an integer and $0\le [r]^+<1$,
$K:=\sup_{t \in [0,T], m=1,2}||v_m (t)||_{B_{p,p}^{r}}$, and $C_1$ is a
positive constant independent of
$K$, $\nu$, $T$, $p'$, $v_m$. In particular, if $r$ is an integer, we use the notation
$||\cdot||_{B_{p,p}^{r}}$ to denote the Sobolev norm
$||\cdot||_{W^{r,p}}$. 
\end{rem}

Note that  estimate (\ref{e24d}) does not yield the regularity with respect to the time variable of
$g$, however, according to
Remark \ref{r1}, we  can show that $g\in C([0,T];
B_{p,p}^{2+\alpha}(\R^d;$ $\R^d))$.

\begin{cor}\label{c1}
Let $v$, $g(t)$ be as  in Lemma \ref{l8}.
Then for every $r>1+\frac{d}{p}$, we have $g\in C([0,T];
B_{p,p}^{r}(\R^d;\R^d))$, and the estimate (\ref{e19}) holds for $g$.
\end{cor}
\begin{proof}
Since $v \in \S(p,p,p',T)$, by the elliptic regularity
$F_v(t) \in C_b^{\infty}(\R^d;\R^d)$. Hence by \cite[Theorem 3.2]{PP}, $g(t,x):=Y_{T-t}^{T-t}(x) \in
C^1([0,T];C_b^{2}(\R^d;\R^d))$ satisfies the following parabolic PDE,
\begin{equation*}
\frac{\partial g}{\partial t}+v \cdot \nabla g = \nu \Delta g+F_v, \ \ g(0)=u_0.
\end{equation*}
Hence for every $0\le t \le s \le T$, $x \in \R^d$,
\begin{equation*}
g(s,x)-g(t,x)=\int_t^s \Big(\nu \Delta g(r,x)+F_v(r,x)-v(r,x) \cdot \nabla g(r,x)\Big) dr,
\end{equation*}
so that, according to (\ref{e19}) and Lemma \ref{l1},
\begin{equation*}
||g(s)-g(t)||_{B_{p,p}^{r}}\le C(s-t)\Big(1+\sup_{t \in [0,T]}
\big(||g(t)||_{B_{p,p}^{r+2}}^2+ ||v(t)||_{B_{p,p}^{r}}^2\big)\Big),
\end{equation*}
which implies that $g\in C([0,T];
B_{p,p}^{r}(\R^d;\R^d))$.
\end{proof}

As stated in Section 2, for every  $v \in \S(p,p,p',T)$, we can define
$\I_{\nu}(v)(t):=\mathbf{P}(Y_{T-t}^{T-t}(.))$ for every
$t \in [0,T]$, where $Y$ is the solution of (\ref{e5}) with coefficients $v$ and initial condition $u_0=v(0)$,
$\mathbf{P}$ is the Leray-Hodge projection to
divergence free vector fields.
By  Lemmas \ref{l8} and \ref{l12},  the extension property of the map $\I_{\nu}$ hold.

\begin{prp}\label{p1}
For every $p>1$, $r>1+\frac{d}{p}$, $T>0$, $u_0 \in B_{p,p}^{r}(\R^d;\R^d)$
satisfying that  $\nabla \cdot u_0=0$,
$\I_{\nu}$ can be extended to be a map
$\I_{\nu}: \B(u_0,T,p,r) \rightarrow \B(u_0,T,p,r)$, where
\begin{equation}\label{e24aa}
\begin{split}
&\B(u_0,T,p,r):=\Big\{v \in C([0,T];B_{p,p}^{r}(\R^d;\R^d));\ \ v(0,x)=u_0(x),\ \ \ \\
&\ \ \nabla \cdot v(t)=0,\ \forall\ t \in [0,T] \Big\},
\end{split}
\end{equation}
and for $v_1, v_2 \in \B(u_0,T,p,r) $, the estimate (\ref{e19aa}), (\ref{e19}) holds
with $g_1,g_2$ replaced by $\I_{\nu}(v_1)$,  $\I_{\nu}(v_2)$.
\end{prp}
\begin{proof}
{\bf Step 1}: Since $C_c^{\infty}(\R^d)$ is dense in
$B_{p,p}^{r}(\R^d)$ (see \cite[Theorem 2.3.2(a)]{T}), for every $v\in \B(u_0,T,p,r)$, we can find a sequence
$\{\tilde v_n\}_{n=1}^{\infty}$, such that for every $n$,
$\tilde v_n \in C([0,T];C_c^{\infty}(\R^d;$ $\R^d))$ (however we can not assume that
$\nabla \cdot \tilde v_n(t)=0$), and
\begin{equation*}
\lim_{n \rightarrow \infty}\sup_{t \in [0,T]}||\tilde v_n(t)-v(t)||_{B_{p,p}^{r}}=0.
\end{equation*}
Let $v_n(t):=\mathbf{P}\tilde v_n(t)$:  we have $v_n \in \S(p,p,p',T)$ for every
$1<p'<\frac{d}{2}$. Also note that $\mathbf{P}$ is a singular integral operator,
which is bounded
in $B_{p,p}^{r}(\R^d)$ (see \cite{S} or the proof of Lemma \ref{l1}). Therefore
\begin{equation*}
\begin{split}
&\lim_{n \rightarrow \infty}\sup_{t \in [0,T]}||v_n(t)-v(t)||_{B_{p,p}^{r}}
=\lim_{n \rightarrow \infty}\sup_{t \in [0,T]}||\mathbf{P} \tilde v_n(t)-\mathbf{P}v(t)||_{B_{p,p}^{r}}\\
&\le C \lim_{n \rightarrow \infty}\sup_{t \in [0,T]}||\tilde v_n(t)-v(t)||_{B_{p,p}^{r}}=0,
\end{split}
\end{equation*}
in particular, for $u_{0,n}:=v_n(0)$,
\begin{equation}\label{e20}
\lim_{n \rightarrow \infty}||u_{0,n}-u_0||_{B_{p,p}^{r}}=0.
\end{equation}
Since $\mathbf{P}$ is a singular integral operator,
by (\ref{e19}) and (\ref{e19aa}), $\{\I_{\nu}(v_n)\}_{n=1}^{\infty}$ is a Cauchy sequence in
$C([0,T];B_{p,p}^{\max(r-1,1)}(\R^d;\R^d))$  and there is a
$\hat v \in C([0,T];B_{p,p}^{\max(r-1,1)}(\R^d;\R^d))$  such that,
\begin{equation}\label{e20a}
\sup_{n}\sup_{t \in [0,T]}||\I_{\nu}(v_n)(t)||_{B_{p,p}^{r}}<\infty,
\end{equation}
\begin{equation}\label{e21}
\lim_{n \rightarrow \infty}\sup_{t \in [0,T]}||\I_{\nu}(v_n)(t)-\hat v(t)||_{B_{p,p}^{\max(r-1,1)}}=0,
\end{equation}
from (\ref{e21}) we know that $\nabla \cdot \hat v(t)=0$ for every $t$. Since
$\I_{\nu}(v_n)(0)=u_{0,n}$ by definition, according to (\ref{e21}) and (\ref{e20}), we have
$\hat v(0)=u_0$. Also note that due to (\ref{e19aa}),
the limit $\hat v$ we have obtained above is independent of
the choice of approximation sequence $\{\tilde v_n\}$ to $v$, so $\I_{\nu}(v):=\hat v$ is well defined.
And by (\ref{e20a}), (\ref{e21}), we obtain immediately that for
every $v_1,v_2 \in \B(u_0,T,p,r)$, (\ref{e19aa}) holds with $g_1$ ,$g_2$ replaced by $\I_{\nu}(v_1)$, $\I_{\nu}(v_2)$.
In order to
prove that $\hat v \in \B(u_0,T,p,r)$, it only remains to show that
$\hat v \in C([0,T];B_{p,p}^{r}(\R^d;\R^d))$.

{\bf Step 2}: For simplicity, we only consider the case $p>d$ and $r=2+\alpha$
for some $0<\alpha<1$, the other case can be shown similarly.
Based on (\ref{e20a}), (\ref{e21}), by the interpolation inequality
\cite[Theorem 2.4.1(a)]{T}, we have,
\begin{equation*}
\lim_{n \rightarrow \infty}
\sup_{t \in [0,T]}||f_n(t)-\hat v(t)||_{W^{2,p}}=0,
\end{equation*}
where $f_n(t):=\I_{\nu}(v_n)(t)$, which implies that
$\hat v \in C([0,T];W^{2,p}(\R^d;\R^d))$.
For every fixed $t \in [0,T]$, taking a subsequence if necessary, we have
\begin{equation*}
\lim_{n \rightarrow \infty}\nabla^2 f_n(t)=\nabla^2 \hat v(t),\ a.e.,
\end{equation*}
where $a.e.$ means the almost everywhere with respect to the Lebesgue measure.
Hence by Fatou lemma,
\begin{equation*}
\begin{split}
& [\hat v(t)]_{B_{p,p}^{2+\alpha}}^p=
\int_{\R^d}\frac{||\nabla^2\hat v(t,\cdot+y)-
\nabla^2\hat v(t,\cdot)||_{L^p}^p}{|y|^{d+\alpha p}}dy\\
&\le \int_{\R^d}\frac{\liminf_{n \rightarrow \infty}||\nabla^2 f_n(t,\cdot+y)-
\nabla^2 f_n(t,\cdot)||_{L^p}^p}{|y|^{d+\alpha p}}dy\\
&\le \liminf_{n \rightarrow \infty}\int_{\R^d}\frac{||\nabla^2 f_n(t,\cdot+y)-
\nabla^2 f_n(t,\cdot)||_{L^p}^p}{|y|^{d+\alpha p}}dy\\
&\le \sup_n \sup_{t \in [0,T]}||f_n(t)||_{B_{p,p}^{2+\alpha}}^p<\infty,
\end{split}
\end{equation*}
so we obtain $\hat v \in L^{\infty}([0,T];B_{p,p}^{2+\alpha}(\R^d;\R^d))$, and for every
$v \in \B(u_0,T,p,r)$, (\ref{e19}) holds with $g$ replaced by $\I_{\nu}(v)$.

{\bf Step 3}: Note that
$v \in C([0,T];B_{p,p}^{2+\alpha}(\R^d;\R^d))$, the backward SDE (\ref{e5}) has a unique solution.
Let $(X_s^t,Y_s^t,Z_s^t)$, $(X_{s,n}^t,Y_{s,n}^t,Z_{s,n}^t)$ be the solution of (\ref{e5}) with the
coefficients $v$ and $v_n$ respectively.
Since $v(t) \in C_b^{1,r(p)}(\R^d;\R^d)$, we know that $X_s^t(\cdot)$ has a version which is differentiable
with respect to $x$, and the derivative $\nabla X_s^t$ satisfies the following,
\begin{equation}\label{e28c}
\begin{split}
& \nabla X_s^t(x)=\mathbf{I}+\int_t^s \nabla v(T-r,X_r^t(x))\nabla X_r^t(x)dr.
\end{split}
\end{equation}
Then following the same arguments of step 1 in the proof of Lemma \ref{l11}
(especially the one for (\ref{e17a})), we derive $\nabla X_s^t \in L_{loc}^p(\R^d;\R^d)$,
\begin{equation}\label{e21aaa}
\lim_{n \rightarrow \infty}\sup_{0\le t \le s \le T}||\nabla X_{s,n}^t-\nabla X_s^t||_{L^p}=0.
\end{equation}
By (\ref{e11a}), $\sup_n \sup_{0\le t \le s \le T}||\nabla^2 X_{s,n}^t||_{B_{p,p}^{\alpha}}<\infty$, which implies that
$\sup_{n,k} \sup_{0\le t \le s \le T}||\nabla X_{s,n}^t-\nabla X_{s,k}^t||_{B_{p,p}^{1+\alpha}}<\infty$ (note that
we only obtain $\nabla X_{s,n}^t- \nabla X_{s,k}^t\in L^p(\R^d;\R^d)$, but
we may not have $\nabla X_{s,n}^t \in L^p(\R^d;\R^d)$); combining this with
(\ref{e21aaa}) and by interpolation inequality, we obtain $\nabla X_s^t \in W_{loc}^{1,p}(\R^d;\R^d)$,
\begin{equation*}
\lim_{n \rightarrow \infty}\sup_{0\le t \le s \le T}||\nabla X_{s,n}^t-\nabla X_s^t||_{W^{1,p}}=0,
\end{equation*}
and we have the following expression,
\begin{equation}\label{e21a}
\begin{split}
& \nabla^2 X_s^t(x)=\int_t^s \nabla v(T-r,X_r^t(x))\nabla^2 X_r^t(x)dr+
\int_t^s \nabla^2 v(T-r,X_r^t(x))\big(\nabla X_r^t(x)\big)^2dr,
\end{split}
\end{equation}
then it is easy to show for every $0\le t \le s \le T$,
\begin{equation*}
\lim_{t' \rightarrow t}||\nabla^2 X_{s}^{t'}-\nabla^2 X_s^t||_{L^p}=0,
\end{equation*}
and following the same procedure as in (\ref{e11a}), we have
$\sup_{0\le t \le s \le T}||\nabla X_{s}^t||_{B_{p,p}^{1+\alpha}}<\infty$.

{\bf Step 4}: Now we want to show that for every $0\le t \le s \le T$,
\begin{equation}\label{e29c}
\begin{split}
& \lim_{t' \rightarrow t}\e\big([\nabla^2 X_s^{t'}-\nabla^2 X_s^t]_{B_{p,p}^{\alpha}}\big)=0.
\end{split}
\end{equation}
We first prove for every $h \in B_{p,p}^{\alpha}(\R^d)$, $0\le t \le s \le T$,
\begin{equation}\label{e28aa}
\begin{split}
& \lim_{t' \rightarrow t}[h \circ X_s^{t'}-h \circ X_s^t]_{B_{p,p}^{\alpha}}=0,\ a.s..
\end{split}
\end{equation}
Without  loss of generality, we assume $0\le t \le t' \le s \le T$. From (\ref{e5}),
let $\Gamma_s^{t,t'}(x):=X_s^t(x)-X_s^{t'}(x)$, $K:=\sup_{t \in [0,T]}||v(t)||_{B_{p,p}^{2+\alpha}}$,
it is easy to see that,
\begin{equation*}
\begin{split}
&|\Gamma_s^{t,t'}(x)| \le |X_{t'}^t(x)-x|+\int_{t'}^s K
|\Gamma_r^{t,t'}(x)|dr\\
&\le K|t'-t|+\sqrt{2\nu}|B_{t'}-B_t|+\int_{t'}^s K
|\Gamma_r^{t,t'}(x)|dr,
\end{split}
\end{equation*}
so by Grownwall lemma, for every $x \in \R^d$,
\begin{equation}\label{e26aa}
\begin{split}
& |\Gamma_s^{t,t'}(x)| \le Ce^{CKT}\big(K|t'-t|+\sqrt{2\nu}|B_{t'}-B_t|\big)
\end{split}
\end{equation}
If $h \in C_c^{\infty}(\R^d)$, since $|X_s^t(x+y)-X_s^t(x)|\le Ce^{CKT}|y|$,
\begin{equation*}
\begin{split}
& \Big|h(X_s^t(x+y))-h(X_s^t(x))-\big(
h(X_s^{t'}(x+y))-h(X_s^{t'}(x))\big)\Big|\\
&\le Ce^{CKT}||\nabla h||_{L^{\infty}}|y|1_{\{|y|\le 1\}}+C
||h||_{L^{\infty}}1_{\{|y|> 1\}}.
\end{split}
\end{equation*}
By (\ref{e26aa}) and the dominated convergence theorem,
\begin{equation*}
\begin{split}
& \lim_{t' \rightarrow t}\int_{\R^d}\int_{\R^d}\frac{\big|h(X_s^t(x+y))-h(X_s^t(x))-\big(
h(X_s^{t'}(x+y))-h(X_s^{t'}(x))\big)\big|^p}{|y|^{d+\alpha p}}dxdy=0,
\end{split}
\end{equation*}
which implies that for every $h \in C_c^{\infty}(\R^d)$,
\begin{equation}\label{e26c}
\lim_{t' \rightarrow t}[h \circ X_s^{t'}-h \circ X_s^t]_{B_{p,p}^{\alpha}}=0.
\end{equation}

Since $C_c^{\infty}(\R^d)$ is dense in $B_{p,p}^{\alpha}(\R^d)$, there exists a sequence
$\{h_n\}_{n=1}^{\infty}\subseteq C_c^{\infty}(\R^d)$, such that
$\lim_{n \rightarrow \infty}||h_n-h||_{B_{p,p}^{\alpha}}=0$. So by the argument in Step 3
of the proof of Lemma \ref{l7}, we obtain,
\begin{equation}\label{e21c}
\lim_{n \rightarrow \infty}\sup_{0\le t \le s \le T}
[h_n \circ X_s^t-h \circ X_s^t]_{B_{p,p}^{\alpha}}=0.
\end{equation}
Combining (\ref{e26c}) and (\ref{e21c}) together,
\begin{equation*}
\begin{split}
& \limsup_{t' \rightarrow t}[h \circ X_s^{t'}-h \circ X_s^t]_{B_{p,p}^{\alpha}}\\
&\le \limsup_{t' \rightarrow t}C[h_n \circ X_s^{t'}-h_n \circ X_s^t]_{B_{p,p}^{\alpha}}
+\limsup_{n \rightarrow \infty}C\sup_{0\le t \le s \le T}
[h_n \circ X_s^t-h \circ X_s^t]_{B_{p,p}^{\alpha}}=0,
\end{split}
\end{equation*}
hence (\ref{e28aa}) holds.

{\bf Step 5}:
Let $\Lambda_s^{t,t'}(x):=\nabla X_s^{t'}(x)-\nabla X_s^{t}(x)$, since
$|\nabla X_s^t(x)|\le e^{CKT}$, by (\ref{e28c}),
\begin{equation*}
\begin{split}
& |\Lambda_s^{t,t'}(x)|\le |\nabla X_{t'}^t(x)-\mathbf{I}|
+K\int_{t'}^s |\Lambda_r^{t,t'}(x)|dr\\
&+Ce^{CKT}\int_{t'}^s \big|\nabla v(T-r, X_{r}^{t'}(x))-
\nabla v(T-r, X_{r}^{t}(x))\big|dr\\
&\le CKe^{CKT}|t'-t|+K\int_{t'}^s |\Lambda_r^{t,t'}(x)|dr
+CKe^{CKT}\int_{t'}^s |\Gamma_r^{t,t'}(x)|^{r(p)}dr,
\end{split}
\end{equation*}
where we also use (\ref{e7d}), and $\Gamma_s^{t,t'}(x):=X_s^t(x)-X_s^{t'}(x)$. So by Grownwall lemma and
(\ref{e26aa}),
\begin{equation}\label{e29aa}
\begin{split}
&|\Lambda_s^{t,t'}(x)|
\le CKe^{CKT}\big((1+T)|t'-t|+\sqrt{2\nu}T|B_{t'}-B_t|\big).
\end{split}
\end{equation}

Let $\Theta_s^{t,t'}(x,y):=\nabla X_s^{t'}(x+y)-\nabla X_s^{t'}(x)-
\big(\nabla X_s^{t}(x+y)-\nabla X_s^{t}(x)\big)$, then by (\ref{e28c}), and applying
(\ref{e22a}), we get,
\begin{equation*}
\begin{split}
& |\Theta_s^{t,t'}(x,y)|\le |J_{t'}^t(x,y)|+\int_{t'}^s K |\Theta_r^{t,t'}(x,y)|dr\\
&+Ce^{CKT}\int_{t'}^s |I_{r}^{t,t'}(x,y)|dr+
CKe^{CKT}\int_{t'}^s |\Lambda_r^{t,t'}(x)|\big(|y|^{r(p)}1_{\{|y|\le 1\}}
+1_{\{|y|> 1\}}\big)dr\\
&+
CK\int_{t'}^s|\Gamma_r^{t,t'}(x+y)|^{r(p)}|J_r^t(x,y)|dr,
\end{split}
\end{equation*}
where we use (\ref{e7d}), and $I_{r}^{t,t'}(x,y)
:=\nabla v(T-r,X_r^{t'}(x+y))-\nabla v(T-r,X_r^{t'}(x))-
\big(\nabla v(T-r,X_r^{t}(x+y))-\nabla v(T-r,X_r^{t}(x))\big)$,
$J_r^t(x,y):=\nabla X_r^t(x+y)-\nabla X_r^t(x)$.
Combining all the estimate above, by Grownwall lemma,
for every $x,y \in \R^d$, $0\le t \le t' \le s \le T$,
\begin{equation*}
\begin{split}
\lim_{t' \rightarrow t}|\Theta_s^{t,t'}(x,y)|=0, \ a.s..
\end{split}
\end{equation*}
By (\ref{e7d}) and Grownwall lemma,
it is not difficult to check that,
\begin{equation*}
\begin{split}
& |I_{r}^{t,t'}(x,y)|\le CKe^{CKT}\big(|y|^{r(p)}1_{\{|y|\le 1\}}
+1_{\{|y|> 1\}}\big),\\
& |J_r^t(x,y)|\le CKe^{CKT}|r-t|\big(|y|^{r(p)}1_{\{|y|\le 1\}}
+1_{\{|y|> 1\}}\big),
\end{split}
\end{equation*}
so
\begin{equation*}
\sup_{0\le t \le t'\le s\le T}|\Theta_s^{t,t'}(x,y)|\le CKe^{CKT}
\big(1+T+|B_{t'}-B_t|\big)
\big(|y|^{r(p)}1_{\{|y|\le 1\}}
+1_{\{|y|> 1\}}\big),
\end{equation*}
hence by the dominated convergence theorem,
\begin{equation}\label{l12.1}
\lim_{t' \rightarrow t} \int_{\R^d}
\int_{\R^d}\frac{|\Theta_s^{t,t'}(x,y)|^p|\nabla^2 v(T-s,X_s^t(x)|^p}{|y|^{d+\alpha p}}dxdy=0.
\end{equation}

By the same procedure of Step 4, we have,
\begin{equation*}
\lim_{t' \rightarrow t} ||\nabla^2 v(T-s,X_s^{t'}(\cdot))-
\nabla^2 v(T-s,X_s^{t}(\cdot))||_{L^p}=0,
\end{equation*}
and by the estimate above for $J_r^t(x,y)$,
\begin{equation}\label{l12.2}
\begin{split}
&\lim_{t' \rightarrow t} \int_{\R^d}\int_{\R^d}\frac{|\nabla^2 v(T-s,X_s^{t'}(x))-
\nabla^2 v(T-s,X_s^{t}(x))|^p |J_s^t(x,y)|^p }{|y|^{d+\alpha p}}dx dy=0.
\end{split}
\end{equation}
By (\ref{e29aa}) and Step 3 in the proof of Lemma \ref{l7},
\begin{equation}\label{l12.3}
\begin{split}
& \lim_{t' \rightarrow t} \int_{\R^d}\int_{\R^d}\frac{|\nabla^2 v(T-s,X_s^{t'}(x+y))-
\nabla^2 v(T-s,X_s^{t'}(x))|^p |\Lambda_s^{t,t'}(x)|^p}{|y|^{d+\alpha p}}dxdy\\
&\le Ce^{CKT}[v(T-s)]_{B_{p,p}^{2+\alpha}}^p\lim_{t' \rightarrow t}\sup_{x \in \R^d}
|\Lambda_s^{t,t'}(x)|^p=0.
\end{split}
\end{equation}
By the inequality (\ref{e22a}), and according to (\ref{e28aa}),
(\ref{l12.1}), (\ref{l12.2}), (\ref{l12.3}),
\begin{equation}\label{e33aa}
\lim_{t' \rightarrow t} \big[\nabla^2 v(T-s,X_s^{t'}(\cdot))\nabla X_s^{t'}(\cdot)
-\nabla^2 v(T-s,X_s^{t}(\cdot))\nabla X_s^{t}(\cdot)\big]_{B_{p,p}^{\alpha}}=0.
\end{equation}

Applying  (\ref{e22a}) to the first term on the right hand side of (\ref{e21a}), in the same way as above to estimate
the associated terms, and using (\ref{e33aa}), Gronwall lemma and the dominated convergence theorem, we have,
\begin{equation*}
\lim_{t' \rightarrow t}[\nabla^2 X_s^{t'}-\nabla^2 X_s^{t}]_{B_{p,p}^{\alpha}}=0,
\end{equation*}
since
\begin{equation*}
\sup_{0\le t\le s \le T}\e\big([\nabla^2 X_s^{t}]_{B_{p,p}^{\alpha}}^p\big)<\infty,
\end{equation*}
$[\nabla^2 X_s^{t'}-\nabla^2 X_s^{t}]_{B_{p,p}^{\alpha}}$ is uniformly integrable, and
(\ref{e29c}) holds.

{\bf Step 6}: By the approximation procedure above, we know $\hat v=\I_{\nu}(v)$ satisfies that
$\hat v(t)=\mathbf{P}(g(t))$, where
\begin{equation*}
g(t,x)=\e\big(u_0(X_T^t(x))\big)+\int_t^T \e\big(F_v(T-s,X_s^t(x))\big)ds,
\end{equation*}
and for every $t \in [0,T]$, $g(t) \in B_{p,p}^{2+\alpha}(\R^d;\R^d)$,
\begin{equation*}
\begin{split}
&\nabla^2 g(t,x)=\e\big(\nabla^2 u_0(X_T^t(x))(\nabla X_T^t(x))^2+
\nabla u_0(X_T^t(x))\nabla^2 X_T^t(x) \big)\\
&+\int_t^T \e\big(\nabla^2 F_v(T-s,X_s^t(x))
(\nabla X_s^t(x))^2+ \nabla F_v(T-s,X_s^t(x))\nabla^2 X_s^t(x)\big)ds,
\end{split}
\end{equation*}
based on such expression, by (\ref{e28aa}), (\ref{e29c}) and by the same methods above, we can show that,
\begin{equation*}
\lim_{t' \rightarrow t}[\hat v(t')-\hat v(t)]_{B_{p,p}^{2+\alpha}}\le
C\lim_{t' \rightarrow t}[g(t')-g(t)]_{B_{p,p}^{2+\alpha}}=0.
\end{equation*}
Since we have shown that $\hat v \in C([0,T];W^{2,p}(\R^d;\R^d))$,
we obtain $\hat v \in C([0,T];B_{p,p}^{2+\alpha}(\R^d;\R^d))$ and the proof is finished.
\end{proof}

\begin{rem}\label{r3.1}
Since for $p>1$, $r>1+\frac{d}{p}$, $v \in C([0,T];B^{r}_{p,p}(\R^d))$ implies that
$v(t)$ and $F_v$ are Lipschitz continuous functions,
 there exists a unique solution $(X,Y,Z)$ for the forward-backward SDE
(\ref{e5}) with coefficients $v$ and initial value $u_0=v(0)$. If we let $g(t):=Y_{T-t}^{T-t}$,
by the approximation procedure in the proof of Proposition \ref{p1}, we know that $\I_{\nu}(v)(t)=\mathbf{P} (g(t))$.
\end{rem}

\begin{thm}\label{t1}
Suppose $p>1$, $r>1+\frac{d}{p}$ and
$u_0 \in  B_{p,p}^{r}(\R^d;\R^d)$ satisfing that $\nabla \cdot u_0=0$;
then there exists a
constant $T_0$, which depends only on $||u_0||_{B_{p,p}^{r}}$ (in particular, $T_0$ is independent of the viscosity
$\nu$), for which  there is a unique fixed
point $u$ of the map $\I_{\nu}$ in $\B(u_0,T_0,p,r)$, where $\B(u_0,T_0,p,r)$
is defined in (\ref{e24aa}).
\end{thm}
\begin{proof}
For each $T>0$
and $v \in \B(u_0,T,p,r)$ with
$\sup_{t \in [0,T]}||v(t)||_{B_{p,p}^{r}}\le K$, by Proposition \ref{p1}, (\ref{e19}) holds with $g$ replaced by
$\I_{\nu}(v)$, i.e.,
\begin{equation*}
\sup_{t \in [0,T]}||\I_{\nu}(v)(t)||_{B_{p,p}^{r}}
\le  C(1+T^{[r]}K^{[r]})e^{CKT}||u_0||_{B_{p,p}^r}
+CTK^2(1+T^{[r]+1}K^{[r]+1}).
\end{equation*}
Note that the above bound in the right hand side tends to
$C||u_0||_{B_{p,p}^{r}}$ as $T$ tends to $0$ and $C$ is independent of $K$. Therefore we can find
constants $K_0>>||u_0||_{B_{p,p}^{r}}$ and $0<T_1<1$ which
only depend on $||u_0||_{B_{p,p}^{r}}$, such that for
every $0<T\le T_1$, $v \in \B(u_0,T,p,r)$ with
$\sup_{t \in [0,T]}||v(t)||_{B_{p,p}^{r}}\le K_0$,
\begin{equation*}
\sup_{t \in [0,T]}||\I_{\nu}(v)||_{B_{p,p}^{r}}\le K_0.
\end{equation*}
Fix such $K_0$; by Proposition \ref{p1}, there is a constant $0<T_0 \le T_1$ only depending on
$||u_0||_{B_{p,p}^{r}}$, such that for
 each $v_1,v_2 \in \B(u_0,T_0,p,r)$ with $||v_m||_{B_{p,p}^{r}}\le K_0,\ \ m=1,2$,
\begin{equation}\label{e23a}
\sup_{t \in [0,T_0]}||\I_{\nu}(v_1)(t)-\I_{\nu}(v_2)(t)||_{B_{p,p}^{r'}}
\le\frac{1}{2}\sup_{t \in [0,T_0]}||v_1(t)-v_2(t)||_{B_{p,p}^{r'}},
\end{equation}
where $r':=\max(r-1,1)$.
For every $v\in C([0,T];B_{p,p}^{r}(\R^d;\R^d))$, let
$||v||_{B_{p,p}^{r},T}:=\sup_{t \in [0,T]}
||v(t)||_{B_{p,p}^{r}}$. From the analysis above, we know that $\I_{\nu}$ can be viewed as a map
$\I_{\nu}:\B(u_0,T_0,p,r,K_0):\rightarrow \B(u_0,T_0,p,r,K_0)$, where
\begin{equation*}
\B(u_0,T_0,p,r,K_0):=\big\{v \in \B(u_0,T_0,p,r);\ \
||v||_{B_{p,p}^{r},T_0}\le K_0\big\},
\end{equation*}
and $\I_{\nu}$ is a contractive map with the $||.||_{B_{p,p}^{r'},T_0}$ norm.

Now we follow the argument in \cite[Theorem 2.1]{I}. Choose  $u_1 \in
\B(u_0,T_0,p,r,K_0)$ (for example, $u_1(t):=u_0$ for every $t\in [0,T_0]$),
and define $u_n:=\I_{\nu}(u_{n-1})$ inductively. Then due to (\ref{e23a}),
\begin{equation*}
||u_{n+1}-u_n||_{B_{p,p}^{r'},T_0}\le \frac{1}{2}
||u_{n}-u_{n-1}||_{B_{p,p}^{r'},T_0},
\end{equation*}
which implies that $\{u_n\}_{n=1}^{\infty}$
has a strong limit $u \in C([0,T_0];B_{p,p}^{r'}(\R^d;\R^d))$ in the
$||.||_{B_{p,p}^{r'},T_0}$ norm.
Since $\sup_n||u_n||_{B_{p,p}^{r},T_0}$ $\le K_0$, as the same procedure in the proof of Proposition \ref{p1},
we have $u \in C([0,T_0];B_{p,p}^{r}(\R^d;\R^d))$ and
$||u||_{B_{p,p}^{r},T_0}\le K_0$. So according to (\ref{e23a}),
\begin{equation*}
||\I_{\nu}(u_n)-\I_{\nu}(u)||_{B_{p,p}^{r'},T_0}\le
\frac{1}{2}||u_{n}-u||_{B_{p,p}^{r'},T_0},
\end{equation*}
which implies that $\I_{\nu}(u)=u$.
The uniqueness of the fixed point $u$ also follows from
(\ref{e23a}).
\end{proof}

Let $C^{1,2}([0,T]\times \R^d;\R^d)$ denote the set of vector fields
which are one order differentiable with respect to the time variable in $[0,T]$ and twice differentiable
with respect to the space variable in $\R^d$. If a solution $u$ of (\ref{e1}) belongs to
$C^{1,2}([0,T]\times \R^d;\R^d)$, it is a classical solution (differentiable in
time and space variables).
Now we prove that the fixed point $u$ of $\I_{\nu}$ is the solution of Navier-Stokes equation (\ref{e1}),
\begin{thm}\label{t2}
Suppose $p>1$, $r>1+\frac{d}{p}$, and
$u_0 \in  B_{p,p}^{r}(\R^d;\R^d)$ satisfing that $\nabla \cdot u_0=0$.
Then there exists a vector field
$u \in C([0,T_0];B_{p,p}^{r}(\R^d;\R^d))$ for
some constant $T_0>0$ which only depends on
$||u_0||_{B_{p,p}^{r}}$ and is independent of $\nu$,
such that $u$ is the unique strong solution of (\ref{e1})
in $C([0,T_0];B_{p,p}^{r}(\R^d;\R^d))$. In particular, if $r>2+\frac{d}{p}$, $u$ is a classical solution of (\ref{e1}).
\end{thm}
\begin{proof}

{\bf Step 1}: Suppose $u_0 \in B_{p,p}^{r}(\R^d;\R^d)$ for general $r>1+\frac{d}{p}$; as  in the proof
of Proposition \ref{p1}, we can obtain
a sequence $\{u_{0,n}\}\subseteq \bigcap_{l>1}B_{p,p}^{l}(\R^d;\R^d)$, such that
$\lim_{n \rightarrow \infty}||u_{0,n}-u_0||_{B_{p,p}^{r}}=0$ and $\nabla \cdot u_{0,n}=0$.
Recall the iteration procedure in the proof of Theorem \ref{t1}; we can find a constant $T_1$ independent of
$n$ and $\nu$,  such that for every $n$, there exist
vectors $\{u_{n,m}\}_{m=1}^{\infty} \subseteq \bigcap_{l>1} C([0,T_1];B_{p,p}^{l}(\R^d;\R^d))$,
$u_n \in C([0,T_1];B_{p,p}^{r}(\R^d;\R^d))$, such that
\begin{equation}\label{e20aa}
\begin{split}
& u_{n,m}(0)=u_{0,n},\ \ u_{n,m+1}=\I_{\nu}(u_{n,m}),\ \ \sup_{m,n}\sup_{t \in [0,T_1]}
||u_{n,m}(t)||_{B_{p,p}^{r}}<\infty,\\
&\lim_{m \rightarrow \infty}\sup_{t \in [0,T_1]}
||u_{n,m}(t)-u_n(t)||_{B_{p,p}^{r'}}=0,\ \ \I_{\nu}(u_n)=u_n,
\end{split}
\end{equation}
where $r':=\max(r-1,1)$.
Moreover, let $(X_{n,m},Y_{n,m},Z_{n,m})$ be the solution of (\ref{e5}) with coefficients
$v=u_{n,m}$ and initial condition $u_{0,n}$.  We define $g_{n,m}(t):=Y_{T_1-t,n,m}^{T_1-t}$ for
$t \in [0,T_1]$; since $u_{n,m}$ is regular enough, $g_{n,m}$ is the unique classical solution of
the following PDE,
\begin{equation}\label{e23c}
\frac{\partial g_{n,m}}{\partial t}+u_{n,m}\cdot \nabla g_{n,m}=\nu \Delta g_{n,m}
+F_{u_{n,m}},\ \ g_{n,m}(0)=u_{0,n}.
\end{equation}
Therefore
it is a strong solution in the following sense, for every $t\in [0,T_1]$,
\begin{equation}\label{e22c}
\begin{split}
& g_{n,m}(t)=e^{t \nu\Delta}u_{0,n}-\int_0^t \Big(e^{(t-s)\nu\Delta}\big(
u_{n,m}(s)\cdot \nabla g_{n,m}(s)-F_{u_{n,m}}(s)\big)\Big)ds.
\end{split}
\end{equation}
Suppose $T_1$ independent of $n$, $m$ small enough, by (\ref{e19aa}) and (\ref{e20aa}) we have,
\begin{equation*}
\sup_{t \in [0,T_1]}||g_{n,m+1}(t)-g_{n,m}(t)||_{B_{p,p}^{r'}}\le 2
\sup_{t \in [0,T_1]}||u_{n,m+1}(t)-u_{n,m}(t)||_{B_{p,p}^{r'}},
\end{equation*}
so by (\ref{e20aa}) and by the same argument in the proof of Proposition \ref{p1}, there is a
$g_n \in C([0,T_1];B_{p,p}^{r}(\R^d;\R^d))$, such that,
\begin{equation}\label{e22aa}
\begin{split}
&\lim_{m \rightarrow \infty}\sup_{t \in [0,T_1]}||g_{n,m}(t)-g_{n}(t)||_{B_{p,p}^{r'}}=0,
\ \ \sup_{n} \sup_{t \in [0,T_1]}||g_{n}(t)||_{B_{p,p}^{r}}<\infty.
\end{split}
\end{equation}
Letting $m \rightarrow \infty$ in (\ref{e22c})  we obtain, for every $t \in [0,T_1]$,
\begin{equation}\label{e23}
\begin{split}
& g_{n}(t)=e^{t \nu\Delta}u_{0,n}-\int_0^t \Big(e^{(t-s)\nu\Delta}\big(
u_{n}(s)\cdot \nabla g_{n}(s)-F_{u_{n}}(s)\big)\Big)ds.
\end{split}
\end{equation}

{\bf Step 2}:
Since $\nabla \cdot u_{n,m}(t)=0$, by
definition (\ref{e5a}) and standard approximation procedure (see Lemma \ref{l1}), for every $t$,
\begin{equation*}
\nabla \cdot F_{u_{n,m}}(t)=\nabla \cdot \big(\nabla N G_{u_{n,m}}(t)\big)=\Delta N G_{u_{n,m}}(t)=
G_{u_{n,m}}(t).
\end{equation*}
Let
\begin{equation*}
H_{u_{n,m},g_{n,m}}(t):=
\sum_{i,j=1}^d\big(
\partial_i u_{n,m}^j(t)\partial_j (g_{n,m}^i(t)-u_{n,m}^i(t))\big).
\end{equation*}

Let $h_{n,m}(t):=\nabla \cdot g_{n,m}(t)$, taking the divergence in (\ref{e23c}), so
for every $t \in [0,T_1]$,
\begin{equation*}
\begin{split}
\frac{\partial h_{n,m}}{\partial t}+u_{n,m}\cdot \nabla h_{n,m}=\nu \Delta h_{n,m}-
H_{u_{n,m},g_{n,m}},\ \ h_{n,m}(0)=0.
\end{split}
\end{equation*}
For every $0\le t \le s \le T_1$,
applying Ito's formula to $h_{n,m}\big(T_1-s, X_{s,n,m}^t(x)\big)$ and taking the
expectation, we get,
\begin{equation*}
h_{n,m}(T_1-t,x)=-\int_t^{T_1}\e\big(H_{u_{n,m},g_{n,m}}(T_1-s,X_{s,n,m}^t(x))\big)ds,
\end{equation*}
hence for every $0\le t \le T_1$,
\begin{equation*}
||h_{n,m}(t)||_{L^p}\le C\int_0^{t}||H_{u_{n,m},g_{n,m}}(s)||_{L^p}ds,
\end{equation*}
so by (\ref{e20aa}), let $m \rightarrow \infty$,
\begin{equation}\label{e22d}
||h_{n}(t)||_{L^p}\le C\int_0^{t}||H_{u_{n},g_{n}}(s)||_{L^p}ds,
\end{equation}
where $h_n(t):=\nabla \cdot g_n(t)$, and
\begin{equation*}
H_{u_{n},g_{n}}(t):=
\sum_{i,j=1}^d\big(
\partial_i u_{n}^j(t)\partial_j (g_{n}^i(t)-u_{n}^i(t))\big).
\end{equation*}

Since for every $v \in C_c^{\infty}(\R^d;\R^d)$, the Leray-Hodge projection has the expression
$v-\mathbf{P}v=\nabla N (\nabla \cdot v)$, for every $p>1$ we have,
\begin{equation*}
\begin{split}
& ||\nabla (v-\mathbf{P}v)||_{L^p}=||\nabla^2 N (\nabla \cdot v)||_{L^p}\\
&\le C||\Delta N (\nabla \cdot v)||_{L^p}=C||\nabla \cdot v||_{L^p},
\end{split}
\end{equation*}
where in the second step we use the elliptic regularity estimate (for example, see \cite{GT})
$||\nabla^2 f ||_{L^p}\le C||\Delta f||_{L^p}$ for every $f\in C_c^{\infty}(\R^d)$. Note that
$\mathbf{P}(g_n(t))=u_n(t)$ as $\mathbf{P}(g_{n,m}(t))=u_{n,m+1}(t)$, by the standard approximation argument,
\begin{equation*}
||\nabla (u_n(t)-g_n(t))||_{L^p}\le C||\nabla \cdot g_n(t)||_{L^p}
=C||h_n(t)||_{L^p},
\end{equation*}
which implies
\begin{equation}\label{e22}
||H_{u_n,g_n}(t)||_{L^p}\le CK||h_n(t)||_{L^p},
\end{equation}
where $K:=\sup_n\sup_{t \in [0,T]}\big(||\nabla u_n(t)||_{L^{\infty}}\big)$.
According to (\ref{e22d}), (\ref{e22}) and Grownwall lemma, we derive
$||h_n(t)||_{L^p}=0$ for every $t\in [0,T_1]$. So
$\nabla \cdot g_n(t)=0$ and $g_n(t)=\mathbf{P}g_n(t)=u_n(t)$.

Since $\I_{\nu}(u_n)=u_n$, by (\ref{e19aa}) and (\ref{e20aa}), there is a $0<T_0\le T_1$
independent of $\nu$, $n$ and a vector $u \in C([0,T_0];B_{p,p}^{r}(\R^d;\R^d))$, such that,
\begin{equation*}
\lim_{n \rightarrow \infty}\sup_{t \in [0,T_0]}||u_{n}(t)-u(t)||_{B_{p,p}^{r'}}
\le 2\lim_{n \rightarrow \infty}||u_{0,n}-u_0||_{B_{p,p}^{r'}}=0,
\end{equation*}
so taking the limit $n \rightarrow \infty$ in (\ref{e23}) we have, for every
$t \in [0,T_0]$,
\begin{equation*}
\begin{split}
& u(t)=e^{t \nu\Delta}u_{0}-\int_0^t \Big(e^{(t-s)\nu\Delta}\big(
u(s)\cdot \nabla u(s)-F_{u}(s)\big)\Big)ds\\
&=e^{t \nu\Delta}u_{0}-\int_0^t \Big(e^{(t-s)\nu\Delta}\big(\mathbf{P}(
u(s)\cdot \nabla u(s))\big)\Big)ds.
\end{split}
\end{equation*}
Hence $u \in C([0,T_0];B_{p,p}^{r}(\R^d;\R^d))$
is the strong solution of
(\ref{e1}) introduced in \cite{FK}.
In particular, if $r>2+\frac{d}{p}$,
by Sobolev embedding theorem, $u$ is a classical solution.

{\bf Step 3}:
Suppose $r>1+\frac{d}{p}$ and $u \in C([0,T_0];B_{p,p}^{r}(\R^d;\R^d))$
 is a strong solution of (\ref{e1}).
Without loss of generality, we assume $T_0$ to be small enough.
Note that under such regularity condition, the backward SDE (\ref{e5})
with coefficients $u$ and initial condition $u_0$ has a unique solution
$(X,Y,Z)$. Let $g(t):=Y_{T_0-t}^{T_0-t}$ for $t \in [0,T_0]$. By (\ref{e19aa}) and the approximation procedure above,
$g \in C([0,T_0];B_{p,p}^{r}(\R^d;\R^d))$
 is the strong solution of following (linear) PDE,
\begin{equation}\label{e1a}
\frac{\partial g}{\partial t}+ u \cdot \nabla g = \nu \Delta g +F_u, \ \ g(0)=u_0.
\end{equation}
 On the other hand, since $u$ is a strong solution of
(\ref{e1}), $u$ is also a strong solution of (\ref{e1a}). Due to the uniqueness of the strong solution
of linear PDE (\ref{e1a}) in such function space, we must have $g(t)=u(t)$, so $u=g=\I_{\nu}(u)$ (see Remark
\ref{r3.1}), hence it is a
fixed point of $\I_{\nu}$ in $\B(u_0,T_0,p,r)$ and it
is unique according to Theorem \ref{t1}.

\end{proof}

\begin{rem}
As we will see in Section 5, in order to prove Theorem \ref{t1}
and \ref{t2}, the estimate for $W^{1,p}$ norm of the difference  is sufficient. Here we prove
the estimate for $B_{p,p}^{1+\alpha}$ norm in Lemma \ref{l12} and, based on such estimate, we can obtain a more
accurate rate for the limit as $\nu \rightarrow 0$ in Section 4.
\end{rem}

\section{The limit to the Euler equation as $\nu \rightarrow 0$ }
From the analysis in Section 3, we know that the maximal time interval $[0,T_0]$ for the local existence
of a solution for (\ref{e1}) is independent of the viscosity $\nu$.
Although when $\nu=0$, the backward SDE in (\ref{e5}) makes no sense, the function
$g(t)$ is still well defined by (\ref{e12a}) since $X_s^t$ here is the solution of an ODE.
Furthermore, the proof of Proposition \ref{p1}, Theorem \ref{t1}, \ref{t2} can still
be applied to the case where $\nu=0$, and for $p>1$, $r>1+\frac{d}{p}$, the  local existence theorem
in Besov space $B_{p,p}^{r}(\R^d;\R^d)$ for the Euler equation (equation (\ref{e1})
with $\nu=0$) can be derived. In this section, we will study the limit behaviour of the solution
of (\ref{e1}) as $\nu \rightarrow 0$.

For every $p>1$, $r>1+\frac{d}{p}$,
$T>0$, $u_0 \in B_{p,p}^{r}(\R^d;\R^d)$, let
$\B(u_0,T,p,r)$ be the set defined by (\ref{e24aa}). For any $\nu \geqslant 0$, $v\in
\B(u_0,T,p,r)$, let $\I_{\nu}:\B(u_0,T,p,r) \rightarrow \B(u_0,T,p,r) $ be the map
constructed in Proposition \ref{p1}.
By Theorem \ref{t1}, given a $u_0 \in B_{p,p}^{r}(\R^d;\R^d)$ with
$\nabla \cdot u_0=0$,
 there is a constant $T_0>0$ independent of
$\nu$ such that for every $\nu \geqslant 0$,  there is a vector
$u_{\nu}$ which is a fixed point of
$\I_{\nu}$ in the space  $\B(u_0,T_0,p,r)$. For not making the notation confusing, we denote
$u_{\nu}$ with $\nu=0$ by $u$, and the initial point is denoted by $u_0$.
Let
$X_{\nu}$, $X$ be the solution of first equation in (\ref{e5}) with coefficients
$u_{\nu}$ and $u$ respectively,  and with the same
driven Brownian motion $B_t$.

In this section, we consider $B_{p,p}^{2+\alpha}$ norm for simplicity, the other cases can
be shown similarly.
We define,
$K:=\sup_{\nu}\sup_{t \in [0,T_0]}$ $||u_{\nu}(t)||_{B_{p,p}^{2+\alpha}}$ $<\infty$
and in the proof of the lemmas in this section, the constant $C$ will change in different line,
but will not depend on the variable stated in the conclusion of the lemmas. We first show the following estimate:
\begin{lem}\label{l4.1}
Suppose $u_0 \in B_{p,p}^{2+\alpha}(\R^d;\R^d)$ for some
$d<p<\infty$, $0<\alpha<1$,
then for every $f_1,f_2 \in W^{1,p}(\R^d)$,
\begin{equation*}
\begin{split}
&\int_{\R^d}|f_1( X_{s,\nu}^t(x))-f_2(X_{s}^t(x))|^p dx\\
&\le C_1||f_1-f_2||_{L^p}^p+  C_1e^{C_1KT_0}||\nabla f_2||_{L^p}^p
\big(T_0\sup_{t \in [0,T_0]}||u_{\nu}(t)-u(t)||_{L^{\infty}}^p+(\sqrt{2\nu}|B_s-B_t|)^p\big)\ a.s.,
\end{split}
\end{equation*}
where $C_1$ is  a positive constant independent of $\nu$, $T_0$, $K$, and $f_m$.
\end{lem}
\begin{proof}
The proof is quite similar to the one of Lemma \ref{l9}.

By the approximation argument, it is enough to prove the conclusion for
every $f_1,f_2 \in C^{1}(\R^d)\bigcap W^{1,p}(\R^d)$.
Since
\begin{equation*}
\begin{split}
&|f_1(X_{s,\nu}^t(x))-f_2( X_{s}^t(x))|\le
|f_1(X_{s,\nu}^t(x))-f_2( X_{s,\nu}^t(x))|+
|f_2(X_{s,\nu}^t(x))-f_2(X_{s}^t(x))|.
\end{split}
\end{equation*}
By (\ref{e10aa}), for every $0\le t \le s \le T$,
\begin{equation}\label{e25aa}
\int_{\R^d}|f_1(X_{s,\nu}^t(x))-f_2(X_{s,\nu}^t(x))|^p dx \le ||f_1-f_2||_{L^p}^p,\ a.s.,
\end{equation}
As in Lemma \ref{l9}, for every $r\in [0,1]$, we define $X_s^{t,r}(x)$ to be the solution of
following SDE,
\begin{equation}\label{e25a}
\begin{cases}
&d X_{s}^{t,r}(x)=r\sqrt{2\nu}dB_s-
u_{r,\nu}(T-s,X_s^{t,r}(x))ds,\\
&X_{t}^{t,r}(x)=x, \ \ 0\le t\le s\le T_0,
\end{cases}
\end{equation}
where $u_{r,\nu}(t,x):=(1-r)u(t,x)+ru_{\nu}(t,x)$.
Clearly we have $X_s^{t,r}(x)=X_s^t(x)$ if $r=0$ and $X_s^{t,x}(x)=X_{s,\nu}^t(x)$ if $r=1$.

Since $u_{\nu}(t)\in B_{p,p}^{2+\alpha}(\R^d;\R^d)$,
$u_{r,\nu}(t)\in C_b^{1,r(p)}(\R^d;\R^d)$, by the argument in \cite{KU}
there is a version of
$X_s^{t,r}(x)$ which is differentiable with $r$, and $V_s^{t,r}(x):=\frac{d}{dr}(X_s^{t,r}(x))$ satisfies the
following SDE,
\begin{equation}\label{e26}
\begin{cases}
&dV_s^{t,r}(x)=\sqrt{2\nu}dB_s-u_{r,\nu}(T-s,X_s^{t,r}(x))V_s^{t,r}(x)ds\\
&+\big(u(T-s,X_{s}^{t,r}(x))-u_{\nu}(T-s,X_{s}^{t,r}(x))\big)ds,\\
&V_t^{t,r}(x)=0,\ \ \  0\le t\le s\le T_0.
\end{cases}
\end{equation}
Comparing with  equation (\ref{e33a}), the martingale part of (\ref{e26}) does not vanish.
By
Grownwall Lemma,
for every
$0\le t\le s \le T_0$, $r\in [0,1]$ and $x\in \R^d$,
\begin{equation}\label{e26a}
|V_s^{t,r}(x)|\le
Ce^{CKT_0}\big(T_0\sup_{t \in [0,T_0]}||u_{\nu}(t)-u(t)||_{L^{\infty}}+
\sqrt{2\nu}|B_s-B_t|\big),\ a.s..
\end{equation}
Also note that $f_2 \in C^1(\R^d)$; then, following the same procedure in Lemma \ref{l9} and especially (\ref{e16a}), we can show that,
\begin{equation*}
\begin{split}
&\int_{\R^d}|f_2(X_{s,\nu}^t(x))-f_2(X_{s}^t(x))|^p dx\\
& \le Ce^{CKT_0}\int_{\R^d}|\nabla f_2(x)|^p dx\big(
T_0\sup_{t \in [0,T_0]}||u_{\nu}(t)-u(t)||_{L^{\infty}}^p+ (\sqrt{2\nu}|B_s-B_t|)^p\big)\ a.s.,
\end{split}
\end{equation*}
together with (\ref{e25aa}), which allows to prove the conclusion.
\end{proof}

\begin{lem}\label{l4.2}
Suppose $u_0$ satisfies the same condition as the one in Lemma
\ref{l4.1}. For every $f_1, f_2 \in B_{p,p}^{1+\alpha}(\R^d)$,
$0\le t \le s \le T_0$,
\begin{equation*}
\begin{split}
& \int_{\R^d}\int_{\R^d}\frac{\Big|\big(f_1(X_{s,\nu}^t(x+y))-f_2(X_{s}^t(x+y))\big)-
\big(f_1(X_{s,\nu}^t(x))-f_2(X_{s}^t(x))\big)\Big|^p}{|y|^{d+p\alpha}} dxdy \\
& \le C_1e^{C_1KT_0}[f_1-f_2]_{B_{p,p}^{\alpha}}^p\\
&+C_1T_0^{p}(1+K^{p})e^{C_1KT_0}
||f_2||_{B_{p,p}^{1+\alpha}}^p
\sup_{t \in [0,T_0]}\big(||u_{\nu}(t)-u(t)||_{B_{p,p}^{1+\alpha}}^p+
(\sqrt{2\nu}|B_t|)^p\big)\ a.s.,
\end{split}
\end{equation*}
where  $C_1$
is a positive constant independent of $\nu$, $K$, $T_0$, $f_m$.
\end{lem}
\begin{proof}
By the approximation argument, it is sufficient to prove the conclusion
for every $f_1,f_2 \in
C^2(\R^d) \bigcap B_{p,p}^{1+\alpha}(\R^d)$.
The proof is almost a repetition of  the steps of the proof of Lemma \ref{l10}, the only difference is that we need to
use the estimate for the solution $V_s^{t,r}(x)$ of (\ref{e26}), rather than that of (\ref{e33a}).

Let $\Gamma_s^{t,r}(x,y):=V_s^{t,r}(x+y)-V_s^{t,r}(x)$. Since the martingale part of
$\Gamma_s^{t,r}(x,y)$ vanishes, we can follow the procedure in the proof
of Lemma \ref{l10} step by step; without loss of generality, we assume
$T_0\le 1$, so by  (\ref{e26a}) we can derive the following estimate similar to (\ref{e13a}),
\begin{equation}\label{e27}
\begin{split}
&|\Gamma_s^{t,r}(x,y)|
\le CT_0(1+K)e^{CKT_0}\sup_{t \in [0,T_0]}
\big(||u_{\nu}(t)-u(t)||_{B_{p,p}^{1+\alpha}}\\
&+\sqrt{2\nu}|B_t|\big)
\big(|y|^{r(p)}
1_{\{|y|\le 1\}}+ 1_{\{|y|>1\}}\big)\ a.s..
\end{split}
\end{equation}
Hence based on the estimate (\ref{e26a}), (\ref{e27}) and following the same steps of Lemma \ref{l10}, we  prove the
conclusion.
\end{proof}

\begin{lem}\label{l4.3}
Suppose $u_0$ satisfies the same conditions in Lemma
\ref{l4.1}. Then, for every $f_1, f_2 \in \bigcap B_{p,p}^{2+\alpha}(\R^d)$ and
$0\le t \le s\le T_0$,
\begin{equation*}
\begin{split}
&||f_1\circ X_{s,\nu}^t(.)-f_2\circ X_{s}^t(.)||_{B_{p,p}^{1+\alpha}}
\le C_1e^{C_1KT_0}(1+T_0K)||f_1-f_2||_{B_{p,p}^{1+\alpha}}\\
&+C_1T_0e^{C_1KT_0}(1+K^2)
||f_2||_{B_{p,p}^{2+\alpha}}\sup_{t \in [0,T_0]}
\big(||u_{\nu}(t)-u(t)||_{B_{p,p}^{1+\alpha}}+\sqrt{2\nu}|B_t|\big)\ a.s.,
\end{split}
\end{equation*}
where $C_1$
is a positive constant independent of $\nu$, $K$, $T$, $f_m$.
\end{lem}
\begin{proof}
By the approximation argument, it is sufficient to prove the conclusion
for every $f_1,f_2 \in
C^2(\R^d) \bigcap B_{p,p}^{1+\alpha}(\R^d)$.
Note that for $\nabla X_{s,\nu}^t(x)-\nabla X_s^t(x)$, the martingale part vanishes,
hence based on Lemma \ref{l4.1} and \ref{l4.2} and repeating the proof of Lemma \ref{l11},
we can prove the conclusion.
\end{proof}

Now we can show the following result about the limit behaviour of $u_{\nu}$.

\begin{thm}\label{t4.1}
Suppose $u_0$ satisfies the same condition as the one in Lemma \ref{l4.1}; then there is a $0<T_1 \le T_0$
(independent of $\nu$),
such that,
\begin{equation}\label{e27a}
\sup_{t \in [0,T_1]}||u_{\nu}(t)-u(t)||_{B_{p,p}^{1+\alpha}}\le \sqrt{2\nu T_1},
\end{equation}
and, for every $0<\beta<\alpha$,
\begin{equation}\label{e27c}
\lim_{\nu \rightarrow 0}\sup_{t \in [0,T_1]}||u_{\nu}(t)-u(t)||_{B_{p,p}^{2+\beta}}=0.
\end{equation}
\end{thm}
\begin{proof}
We define,
\begin{equation}\label{e25}
\begin{split}
&g_{\nu}(T_0-t,x):=\e\big(u_0(X_{T_0,\nu}^t(x))\big)+\int_t^{T_0} \e\big(F_{u_{\nu}}(T_0-s,X_{s,\nu}^t(x))\big)ds,\\
&g(T_0-t,x):=\e\big(u_0(X_{T_0}^t(x))\big)+\int_t^{T_0} \e\big(F_u(T_0-s,X_s^t(x))\big)ds.
\end{split}
\end{equation}
By Lemma \ref{l3}, Lemma \ref{l4.3} and H\"older  inequality,
\begin{equation*}
\begin{split}
& ||\e\big(F_{u_{\nu}}(T_0-s,X_{s,\nu}^t(\cdot ))-F_u(T_0-s,X_s^t(\cdot ))\big)||_{B_{p,p}^{1+\alpha}}^p\\
&\le \e\big(||
F_{u_{\nu}}(T_0-s,X_{s,\nu}^t(\cdot ))-F_u(T_0-s,X_s^t(\cdot ))||_{B_{p,p}^{1+\alpha}}^p\big)\\
&\le  Ce^{CKT_0}(1+T_0^pK^p)||F_{u_{\nu}}(T-s)-F_u(T-s)||_{B_{p,p}^{1+\alpha}}^p\\
&+CT_0^pe^{CKT_0}(1+K^{2p})
||F_u(T-s)||_{B_{p,p}^{2+\alpha}}^p\big(\sup_{t \in [0,T_0]}
||u_{\nu}(t)-u(t)||_{B_{p,p}^{1+\alpha}}^p+(\sqrt{2\nu T_0})^p\big).
\end{split}
\end{equation*}
Hence by Lemma \ref{l1} and \ref{l3},
\begin{equation*}
\begin{split}
& ||\e\big(F_{u_{\nu}}(T_0-s,X_{s,\nu}^t(\cdot ))-F_u(T_0-s,X_s^t(\cdot ))\big)||_{B_{p,p}^{1+\alpha}}^p\\
&\le  CK^pe^{CKT_0}\big(1+T_0^{p}K^{p}+T_0^pK^p(1+K^{2p})\big)
\big(\sup_{t \in [0,T_0]}
||u_{\nu}(t)-u(t)||_{B_{p,p}^{1+\alpha}}^p+(\sqrt{2\nu T_0})^p\big).
\end{split}
\end{equation*}
Analogously,
\begin{equation*}
\begin{split}
& ||\e\big(u_0(X_{T_0,\nu}^t(\cdot))-u_0(X_{T_0}^t(\cdot))\big)||_{B_{p,p}^{1+\alpha}}^p\\
&\le CT_0^{p}K^pe^{CKT_0}(1+K^{2p})\big(\sup_{t \in [0,T_0]}
||u_{\nu}(t)-u(t)||_{B_{p,p}^{1+\alpha}}^p+(\sqrt{2\nu T_0})^p\big).
\end{split}
\end{equation*}
Combining this estimate and (\ref{e25}), and noting that
$u_{\nu}(t)-u(t)=\mathbf{P}(g_{\nu}(t)-g(t))$, we have,
\begin{equation*}
\sup_{t \in [0,T_0]}
||u_{\nu}(t)-u(t)||_{B_{p,p}^{1+\alpha}}\le
CT_0 e^{CKT_0}(1+K^3)\big(\sup_{t \in [0,T_0]}
||u_{\nu}(t)-u(t)||_{B_{p,p}^{1+\alpha}}+\sqrt{2\nu T_0}\big).
\end{equation*}
Choosing $0<T_1\le T_0$, such that $CT_1e^{CKT_1}(1+K^3)\le \frac{1}{2}$, we obtain,
\begin{equation*}
\sup_{t \in [0,T_1]}
||u_{\nu}(t)-u(t)||_{B_{p,p}^{1+\alpha}}\le  \sqrt{2\nu T_1},
\end{equation*}
which implies (\ref{e27a}).

For every $0<\beta<\alpha$, due the interpolation inequality in
\cite[Theorem 2.4.1(a)]{T}, we obtain,
\begin{equation*}
||u_{\nu}(t)-u(t)||_{B_{p,p}^{2+\beta}}\le C||u_{\nu}(t)-u(t)||_{B_{p,p}^{1+\alpha}}^{\theta}
||u_{\nu}(t)-u(t)||_{B_{p,p}^{2+\alpha}}^{1-\theta},
\end{equation*}
where $0<\theta<1$ is the unique number such that $\theta(1+\alpha)+(1-\theta)(2+\alpha)=2+\beta$. (i.e. $\theta=\beta-\alpha$).
Since
$$\sup_{\nu}||u_{\nu}(t)-u(t)||_{B_{p,p}^{2+\alpha}}\le 2\sup_{\nu}\sup_{t \in [0,T_0]}
||u_{\nu}(t)||_{B_{p,p}^{2+\alpha}}\le 2K, $$
and, according to (\ref{e27a}), we can show (\ref{e27c}).
\end{proof}
\begin{rem}\label{r4.1}
As stated in Remark \ref{r1}, by the same methods above, we can estimate  lower and higher order
Besov norms. More precisely, if $u_0 \in B_{p,p}^{r}(\R^d;\R^d)$ for some $p>1$, $r>1+\frac{d}{p}$,
there exists a constant $0<T_1\le T_0$ such that,
\begin{equation*}
\sup_{t \in [0,T_1]}||u_{\nu}(t)-u(t)||_{B_{p,p}^{\max(r-1,1)}}\le \sqrt{2 \nu T_1}
\end{equation*}
and, for every $0<\tilde r<r$,
\begin{equation*}
\lim_{\nu \rightarrow 0}\sup_{t \in [0,T_1]}||u_{\nu}(t)-u(t)||_{B_{p,p}^{\tilde r}}=0.
\end{equation*}
\end{rem}

\section{The local existence theorem in $B_{p,q}^{r}(\R^d;\R^d)$}
As pointed out in Section 3, if we use the "Lagrangian path" (forward equation)
in (\ref{e3}), then we are unable to derive similar estimates for $B_{p,q}^{r} (p \neq q)$ norms. In this section we will
adopt a different "Lagrangian path", which is just a translation by a Brownian motion; together with the associated
forward-backward stochastic differential system similar to (\ref{e5}), we can establish useful estimates
for $B_{p,q}^{r}$ norms, which however depend on the viscosity
$\nu$. Therefore they can not be applied to the case of $\nu=0$, i.e. to the Euler equation.

As  in  Section 2, by  It\^o's formula and the theorem of
backward SDEs, $u$ is a (regular enough) solution of (\ref{e1}) in time interval $[0,T]$, if and only if
$(X_s^t(x),Y_s^t(x),Z_s^t(x),$ $u(t,x), p(t,x))$ satisfies the following (forward-) backward
stochastic differential system,
\begin{equation}\label{e28a}
\begin{cases}
& dX_s^t(x)=\sqrt{2\nu}dB_s\\
& dY_s^t(x)=\sqrt{2\nu}Z^{t}_s(x)dB_s+
u(T-s,X_s^t(x))\cdot Z_s^t(x)+\nabla p(T-s, X_s^t(x))ds\\
& Y_t^t(x)=u(T-t,x),\  \Delta p(t,x)=-\sum_{i,j=1}^3\partial_i u^j(t,x)\partial_j u^i(t,x)\\
& X_t^t(x)=x, Y_T^t(x)=u_0(X^t_T(x)).
\end{cases}
\end{equation}

For every $v \in \S(p,q,p',T)$ with
$v(0)=u_0$ for some $1<p<\infty$, $1\le q \le \infty$, $1<p'<\frac{d}{2}$, we consider the following
(forward-) backward SDE,
\begin{equation}\label{e29}
\begin{cases}
& dX_s^t(x)=\sqrt{2\nu} dB_s\\
& dY_s^t(x)=\sqrt{2\nu}Z^{t}_s(x)dB_s+
v(T-s,X_s^t(x))\cdot Z_s^t(x)-F_v(T-s, X_s^t(x))ds\\
& X_t^t(x)=x, Y_T^t(x)=u_0(X^t_T(x)),
\end{cases}
\end{equation}
where the vector $F_v$ is defined by (\ref{e5a}).

We first cite the following well-known Bismut-Elworthy-Li formula, e.g. see \cite{Bis}, \cite{EL},
\begin{lem}\label{l5.0}
Let $X_s^t(x)=x+\sqrt{2\nu}(B_s-B_t)$ for every $0\le t < s \le T$. Then for each
$f \in C_b(\R^d)$,
\begin{equation}\label{e45}
\begin{split}
& \nabla \big(\e(f(X_s^t(x)))\big)=\e\big(\nabla f (X_s^t(x))\big)
=\frac{1}{\sqrt{2\nu}(s-t)}\e\big(f (X_s^t(x))(B_s-B_t)\big).
\end{split}
\end{equation}
\end{lem}

We have the following estimate,
\begin{lem}\label{l5.1}
Suppose $v\in \S(p,q,p',T)$
for some $1<p<\infty$, $1\le q \le \infty$, $1<p'<\frac{d}{2}$, $0<T<1$. Let
$(X,Y,Z)$ be the unique solution of (\ref{e29}) with coefficients $v$ and initial condition
$u_0=v(0)$, and let $g(t,x):=Y_{T-t}^{T-t}(x)$. Then for every $0<\alpha<1$,
$\max(1,2-\alpha)<\beta<2$, there exists a $0<T_0\le 1$ which only depends on
$K_1,K_{2,\beta},\nu$, such that for every $0<T<T_0$,
\begin{equation}\label{e24}
\begin{split}
&||g||_{1,T} \le
C_1||u_0||_{B_{p,q}^{1+\alpha}}+C_1K_1K_{2,\beta}T^{\frac{\beta-1}{\beta}},\\
&||g||_{2,\beta,T}\le C_1||u_0||_{B_{p,q}^{1+\alpha}}\big(\nu^{-\frac{1}{2}}+T^{\frac{1}{2}}\big)
+CT^{\frac{2\alpha+3\beta-4}{2\beta}}(K_1^2+K_{2,\beta}^2),
\end{split}
\end{equation}
where $||g||_{1,T}:=\sup_{t \in [0,T]}||g(t)||_{B_{p,q}^{1+\alpha}}$,
$||g||_{2,\beta,T}:=\big(\int_0^T||g(t)||_{B_{p,q}^{2+\alpha}}^{\beta}dt\big)^{\frac{1}{\beta}}$,
$K_1:=\sup_{t \in [0,T]}$ $||v(t)||_{B_{p,q}^{1+\alpha}}$,
$K_{2,\beta}:=\big(\int_0^T||v(t)||_{B_{p,q}^{2+\alpha}}^{\beta}dt\big)^{\frac{1}{\beta}}$,
and $C_1$ is a constant independent of
$K_1$, $K_2$, $\nu$, $T$, $p'$ and $v$.
\end{lem}
\begin{proof}


Some parts of the proof are inspired by reference \cite{K}.

{\bf Step 1}: Let $K_1(t):=||v(t)||_{B_{p,q}^{1+\alpha}}$,
$K_2(t):=||v(t)||_{B_{p,q}^{2+\alpha}}$,
$||g||_{3,T}:=\sup_{t \in (0,T]}\{t^{\frac{1}{2}}$ $||g(t)||_{B_{p,q}^{2+\alpha}}\}$.
Since $||g||_{2,\beta,T}\le C||g||_{3,T}$, it is sufficient to prove
(\ref{e24}) for $||g||_{3,T}$.

Since $v \in \S(p,q,p',T)$, by \cite[Theorem 2.9]{PP}, we can find a version of
$(X_s^t(x), Y_s^t(x),$ $Z_s^t(x))$ which is a.s.  differentiable (for any order) with respect to $x$
and, according to
\cite{PP}, we know that  for every $0\le t \le T$, $l>0$,
\begin{equation}\label{e29a}
\sum_{k=0}^2
\e\Big(\sup_{t\le s \le T }\big(|\nabla^k Y_s^t(x)|^l+|\nabla^k Z_s^t(x)|^l\big)\Big)<\infty.
\end{equation}
From \cite[Lemma 2.5]{PP}, for $g(t,x):=Y_{T-t}^{T-t}(x)$
we have $Y_s^t(x)=g(T-s,X_s^t(x))$,
$Z_s^t(x)=\nabla Y_s^t(x)=\nabla g(T-s,X_s^t(x))$ for every $0\le t \le s \le T$ (note that
the $Z_s^t(x)$ in this paper is actually $\frac{Z_s^t(x)}{\sqrt{\nu}}$ in \cite{PP}).


Taking the expectation in (\ref{e29}), and taking the derivative in $x$, for every $0\le t \le T$,
\begin{equation}\label{e30}
\begin{split}
&\nabla g(T-t,x)=\e\big(\nabla g(T-t,X_t^t(x))\big)\\
&=\e\big(\nabla u_0(X_T^t(x))\big)-\int_t^T
\e\Big(\nabla \big(v \nabla g\big)(T-r,X_r^t(x))\Big)dr
+\int_t^T \e\big(\nabla F_v(T-r,X_r^t(x))\big)dr\\
&:=I_0^t(x)+\sum_{i=1}^2\int_t^T I_{r,i}^t(x) dr.
\end{split}
\end{equation}

Since $X_s^t(x)=x+B_s-B_t$, it is easy to check that for every
$f \in C_c^{\infty}(\R^d)$, $l>1$, $y \in \R^d$,
\begin{equation}\label{e35a}
\begin{split}
& \int_{\R^d}f(X_s^t(x))dx=\int_{\R^d}f(x)dx,\ a.s.,
\end{split}
\end{equation}
\begin{equation}\label{e39}
||f(X_s^t(\cdot+y))-f(X_s^t(\cdot))||_{L^l}=||f(\cdot+y)-f(\cdot)||_{L^l},\ a.s..
\end{equation}
We first consider the case $1\le q<\infty$. For every $y\in \R^d$, by (\ref{e39}) and H\"older inequality,
\begin{equation*}
\begin{split}
&\int_{\R^d}\frac{||I_0^t(\cdot+y)-I_0^t(\cdot)||_{L^p}^q}{|y|^{3+q\alpha}}dy\\
&\le \int_{\R^d}
\frac{\Big(\int_{\R^d}\e\big(\big|\nabla u_0(X_T^t(x+y))-\nabla u_0(X_T^t(x))\big|^p
\big)dx\Big)^{\frac{q}{p}}}{|y|^{3+q\alpha}}dy\\
&\le \int_{\R^d}\frac{||\nabla u_0(\cdot+y)-\nabla u_0(\cdot)||_{L^p}^q}{|y|^{3+q\alpha}}dy
\le ||u_0||_{B_{p,q}^{1+\alpha}}^q.
\end{split}
\end{equation*}
Then we have,
\begin{equation*}
||I_0^t(\cdot)||_{B_{p,q}^{\alpha}}\le C||u_0||_{B_{p,q}^{1+\alpha}}.
\end{equation*}
By (\ref{e45}),
\begin{equation}\label{e30aa}
I_{r,1}^t(x)=-\frac{1}{\sqrt{2\nu}(r-t)}\e\Big(\big((v \nabla g)(T-r,X_r^t(x))\big)\big(B_r-B_t\big)\Big).
\end{equation}
According to (\ref{e7d}), we have,
\begin{equation}\label{e30a}
\begin{split}
& \int_{\R^d}\e\Big(\big|(v\nabla g)(T-r,X_r^t(x+y))-
(v\nabla g)(T-r,X_r^t(x))\big|^p\Big)dx\\
&\le CK_1(T-r)^p||\nabla g(T-r,\cdot+y)-\nabla g(T-r,\cdot)
||_{L^p}^p\\
&+C||v(T-r,X_r^t(\cdot+y))-
v(T-r,X_r^t(\cdot))||_{L^{\infty}}^p||\nabla g(T-r)||_{L^p}^p\\
&\le CK_1(T-r)^p||\nabla g(T-r,\cdot+y)-\nabla g(T-r,\cdot)
||_{L^p}^p\\
&+CK_1(T-r)^p\big(|y|^{pr(p)}1_{\{|y|\le 1\}}+ 1_{\{|y|> 1\}}\big)||g(T-r)||_{W^{1,p}}^p,
\end{split}
\end{equation}
and by H\"older inequality,
\begin{equation}\label{e36aa}
\begin{split}
& \int_{\R^d}\frac{||I_{r,1}^t(\cdot+y)-I_{r,1}^t(\cdot)||_{L^p}^q}{|y|^{3+q\alpha}}dy\\
&\le \int_{\R^d}\frac{\big|\e\big(|B_r-B_t|^{\frac{p}{p-1}}\big)\big|^{\frac{q(p-1)}{p}}}{\nu^{\frac{q}{2}}(r-t)^{q}
|y|^{3+q\alpha}}
\cdot\Big(\int_{\R^d}\e\Big(\big|(v\nabla g)(T-r,X_r^t(x+y))-
(v\nabla g)(T-r,X_r^t(x))\big|^p\Big)dx\\
&\le  \frac{CK_1^q}{(\nu(r-t))^{\frac{q}{2}}}||g(T-r)||_{B_{p,q}^{1+\alpha}}^q.
\end{split}
\end{equation}
Then we get,
\begin{equation*}
||I_{r,1}^t(\cdot)||_{B_{p,q}^{\alpha}}\le \frac{CK_1}
{\nu^{\frac{1}{2}}(r-t)^{\frac{1}{2}}}||g(T-r)||_{B_{p,q}^{1+\alpha}}.
\end{equation*}
Similarly, by (\ref{e6}), H\"older inequality and noting that $||\nabla v(t)||_{L^{\infty}}\le CK_2(t)$,
\begin{equation*}
\begin{split}
& \int_t^T ||I_{r,2}^t(\cdot)||_{B_{p,q}^{\alpha}}dr\le
CK_1\int_0^T K_{2}(s)ds \le CK_1K_{2,\beta}T^{\frac{\beta-1}{\beta}}.
\end{split}
\end{equation*}
As above, taking the expectation of (\ref{e29}), by (\ref{e6}), we deduce that
\begin{equation*}
\begin{split}
& ||g(T-t)||_{L^p}\le C||u_0||_{L^p}+K_1\int_t^T ||g(T-r)||_{W^{1,p}}dr+
K_1\int_t^T K_{2}(T-r)dr.
\end{split}
\end{equation*}
Combining all estimates we have established above and (\ref{e30}), we obtain
\begin{equation}\label{e38}
\begin{split}
&||g||_{1,T}
\le C||u_0||_{B_{p,q}^{1+\alpha}}+C
\big(T^{\frac{1}{2}}\nu^{-\frac{1}{2}}+T\big)K_1||g||_{1,T}
+CK_1K_{2,\beta}T^{\frac{\beta-1}{\beta}}.
\end{split}
\end{equation}

{\bf Step 2}: Taking the derivative with respect to $x$ in (\ref{e30}), to obtain
\begin{equation}\label{e31aa}
\begin{split}
&\nabla^2 g(T-t,x)
=\e\big(\nabla^2 u_0(X_{T}^t(x))\big)-\int_t^{T}
\e\Big(\nabla \big(v \nabla^2 g\big)(T-r,X_r^t(x))\Big)dr\\
&-\int_t^{T} \e\big((\nabla^2 v \nabla g)(T-r,X_r^t(x))\big)dr
-\int_t^{T} \e\big((\nabla v \nabla^2 g)(T-r,X_r^t(x))\big)dr\\
&+\int_t^{T} \e\big(\nabla F_v(T-r,X_r^t(x))\big)dr\\
&:=J_{0}^t(x)+\sum_{i=1}^4\int_t^{T} J_{r,i}^t(x) dr.
\end{split}
\end{equation}

According to (\ref{e45}),
\begin{equation*}
\begin{split}
& J_{0}^t(x)=\frac{1}{\sqrt{2\nu}(T-t)}\e\Big(\nabla u_0(X_{T}^t(x))\big(B_{T}-B_t\big)\Big),\\
& J_{r,1}^t(x)=-\frac{1}{\sqrt{2\nu}(r-t)}\e\Big(\big((v \nabla^2 g)(T-r,X_r^t(x))\big)\big(B_r-B_t\big)\Big);
\end{split}
\end{equation*}
analogously to (\ref{e30a}), (\ref{e36aa}) and using (\ref{e7d}) we obtain
\begin{equation*}
\begin{split}
& ||J_{0}^t(\cdot)||_{B_{p,q}^{\alpha}}\le \frac{C||u_0||_{B_{p,q}^{1+\alpha}}}{\nu^{\frac{1}{2}}(T-t)^{\frac{1}{2}}}
,\\
&||J_{r,1}^t(\cdot)||_{B_{p,q}^{\alpha}}\le \frac{CK_1||g(T-r)||_{B_{p,q}^{2+\alpha}}}{\sqrt{2\nu(r-t)}}
\le \frac{CK_1||g||_{3,T}}{\nu^{\frac{1}{2}}(r-t)^{\frac{1}{2}}(T-r)^{\frac{1}{2}}},\\
&||J_{r,2}^t(\cdot)||_{B_{p,q}^{\alpha}}\le C\big(
||g(T-r)||_{B_{p,q}^{2+\alpha}}||v(T-r)||_{W^{2,p}}+
||\nabla g(T-r)||_{L^{\infty}}||v(T-r)||_{B_{p,q}^{2+\alpha}}\big),\\
&||J_{r,3}^t(\cdot)||_{B_{p,q}^{\alpha}}\le C\big(
||g(T-r)||_{B_{p,q}^{2+\alpha}}||\nabla v(T-r)||_{L^{\infty}}+
||g(T-r)||_{W^{2,p}}||v(T-r)||_{B_{p,q}^{2+\alpha}}\big);
\end{split}
\end{equation*}
and by the interpolation inequality in \cite[Theorem 2.4.1]{T}, we conclude that
\begin{equation}\label{e32aa}
\begin{split}
&||\nabla g(T-r)||_{L^{\infty}}\le C||g(T-r)||_{W^{2,p}}\le C||g(T-r)||^{\alpha}_{B_{p,q}^{1+\alpha}}
||g(T-r)||^{1-\alpha}_{B_{p,q}^{2+\alpha}},\\
&||\nabla v(T-r)||_{L^{\infty}}\le C||v(T-r)||_{W^{2,p}}\le C K_1(T-r)^{\alpha}K_{2}(T-r)^{1-\alpha}.
\end{split}
\end{equation}
Hence by H\"older inequality, for any $\max(1,2-\alpha)<\beta<2$,
\begin{equation*}
\begin{split}
&\int_t^T||J_{r,2}^t(\cdot)||_{B_{p,q}^{\alpha}}dr\le CK_1^{\alpha}||g||_{3,T}\int_0^T
\frac{K_{2}(s)^{1-\alpha}}{s^{\frac{1}{2}}}ds+C||g||_{1,T}^{\alpha}||g||_{3,T}^{1-\alpha}\int_0^T
\frac{K_{2}(s)}{s^{\frac{1-\alpha}{2}}}ds,\\
&\le CK_1^{\alpha}K_{2,\beta}^{1-\alpha}||g||_{3,T}T^{\frac{2\alpha+\beta-2}{2\beta}}+
CK_{2,\beta}||g||_{1,T}^{1-\alpha}||g||_{3,T}^{\alpha}T^{\frac{\alpha\beta+\beta-2}{2\beta}},
\end{split}
\end{equation*}
and
\begin{equation*}
\begin{split}
&\int_t^T||J_{r,3}^t(\cdot)||_{B_{p,q}^{\alpha}}dr\le
CK_1^{\alpha}K_{2,\beta}^{1-\alpha}||g||_{3,T}T^{\frac{2\alpha+\beta-2}{2\beta}}+
CK_{2,\beta}||g||_{1,T}^{\alpha}||g||_{3,T}^{1-\alpha}T^{\frac{\alpha\beta+\beta-2}{2\beta}}.
\end{split}
\end{equation*}

Similarly, by (\ref{e6}) and (\ref{e32aa}),
\begin{equation*}
\int_t^T||J_{r,4}^t(\cdot)||_{B_{p,q}^{\alpha}}dr
\le CK_1^{\alpha}\int_0^T K_{2}(s)^{2-\alpha}ds
\le CK_1^{\alpha}K_{2,\beta}^{2-\alpha}T^{\frac{\alpha+\beta-2}{\beta}}.
\end{equation*}
Since $\int_t^T\frac{1}{(r-t)^{\frac{1}{2}}(T-r)^{\frac{1}{2}}}dr=B(\frac{1}{2}, \frac{1}{2})$,
we have,
\begin{equation*}
\int_t^T||J_{r,1}^t(\cdot)||_{B_{p,q}^{\alpha}}dr.
\le CK_1||g||_{3,T}.
\end{equation*}
Using the inequality
$a^{\alpha}b^{1-\alpha}\le C(a+b)$ for every $a,b>0$, and putting  all above estimates
together into (\ref{e31aa}), we obtain,
\begin{equation*}
\begin{split}
& ||\nabla^2 g(T-t)||_{B_{p,q}^{\alpha}}\le \frac{C||u_0||_{B_{p,q}^{1+\alpha}}}{\nu^{\frac{1}{2}}(T-t)^{\frac{1}{2}}}
+\frac{CK_1||g||_{3,T}}{\nu^{\frac{1}{2}}}\\
&+CT^{\frac{\alpha+\beta-2}{\beta}}\big((K_1+K_{2,\beta})||g||_{3,T}+K_{2,\beta}||g||_{1,T}+K_1^2+K_{2,\beta}^2\big),
\end{split}
\end{equation*}
where we have used the fact that $T^{a_1}\le T^{a_2}$ for $0<a_2\le a_1$ as
$T \le 1$. Combining this with (\ref{e38}), we have
\begin{equation}\label{e37}
\begin{split}
& ||g||_{3,T}\le \sup_{t \in [0,T)}\{(T-t)^{\frac{1}{2}}|| g(T-t)||_{B_{p,q}^{2+\alpha}}\}
\le C||u_0||_{B_{p,q}^{1+\alpha}}\big(\nu^{-\frac{1}{2}}+T^{\frac{1}{2}}\big)
+CK_1T^{\frac{1}{2}}\nu^{-\frac{1}{2}}\big(||g||_{1,T}\\
&+||g||_{3,T}\big)
+CT^{\frac{2\alpha+3\beta-4}{2\beta}}\big((K_1+K_{2,\beta})||g||_{3,T}+(K_1+K_{2,\beta})||g||_{1,T}+K_1^2+K_{2,\beta}^2\big).
\end{split}
\end{equation}
From (\ref{e38}), (\ref{e37}), if we take $0<T_0<1$ which only depends
on $K_1$, $K_{2,\beta}$, $\nu$, such that
\begin{equation*}
C\big(T_0^{\frac{1}{2}}\nu^{-\frac{1}{2}}+T_0\big)K_1\le \frac{1}{4},\
CK_1T_0^{\frac{1}{2}}\nu^{-\frac{1}{2}}\le \frac{1}{4},\
CT_0^{\frac{2\alpha+3\beta-4}{2\beta}}(K_1+K_{2,\beta})\le \frac{1}{4},
\end{equation*}
then  estimate (\ref{e24}) holds.

If $q=\infty$, by the same argument as above, we may still prove (\ref{e24}).

\end{proof}

We also have the following difference estimate.

\begin{lem}\label{l5.1a}
Suppose that $v_m\in \S(p,q,p',T),\ m=1,2$,
for some $d<p<\infty$, $1\le  q \le \infty$, $1<p'<\frac{d}{2}$, $0<T<1$, let
$(X_m,Y_m,Z_m)$ be the unique solution of (\ref{e5}) with coefficients $v_m$ and initial condition
$u_{0,m}:=v_m(0)$, and let $g_m(t,x):=Y_{T-t,m}^{T-t}(x)$. Then for every $0<\alpha<1$, $\max(1,2-\alpha)<\beta<2$,
there is a constant $0<T_0\le 1$ which only depends on $K_1$, $K_{2,\beta}$, $\nu$, such that
for every $0<T\le T_0$,
\begin{equation}\label{e38a}
\begin{split}
&\sup_{t \in [0,T]}||g_1(t)-g_2(t)||_{W^{1,p}}
\le C_1||u_{0,1}-u_{0,2}||_{W^{1,p}}\\
&+C_1(1+K_1^2+K_{2,\beta}^2)T^{\frac{\beta-1}{\beta}}(1+\nu^{-1})\sup_{t \in [0,T]}
||v_1(t)-v_2(t)||_{W^{1,p}},
\end{split}
\end{equation}
where $K_1:=\sup_{m=1,2}\sup_{t \in [0,T]}||v_m(t)||_{B_{p,q}^{1+\alpha}}$,
 $K_{2,\beta}:=\sup_{m=1,2}\big(\int_0^T||v_m(t)||_{B_{p,q}^{2+\alpha}}^{\beta}dt\big)^{\frac{1}{\beta}}$,
$C_1$ is a constant independent of
$K_1$, $K_{2,\beta}$, $\nu$, $v_m$, $p'$ and $T$.
\end{lem}
\begin{proof}
Note that for different $v_1,v_2$, the forward equation in (\ref{e29}) is the same, and
for every $f_1,f_2 \in L^p(\R^d)$  we have $||f_1(X_s^t(\cdot))-f_2(X_s^t(\cdot))||_{L^p}
=||f_1-f_2||_{L^p}$. As in (\ref{e30}), we define,
\begin{equation*}
\nabla g_m(T-t,x):=I^t_{0,m}(x)+\sum_{i=1}^2\int_t^T I_{r,i,m}^t(x)dr,\ \ m=1,2.
\end{equation*}
It is clear that
\begin{equation*}
||I^t_{0,1}(\cdot)-I^t_{0,2}(\cdot)||_{L^p}\le C||u_{0,1}-u_{0,2}||_{W^{1,p}}.
\end{equation*}
By (\ref{e45}),
\begin{equation*}
I_{r,1,m}^t(x)=-\frac{1}{\sqrt{2\nu}(r-t)}\e\Big(\big((v_m \nabla g_m)(T-r,X_r^t(x))\big)\big(B_r-B_t\big)\Big),
\end{equation*}
and using H\"older inequality as in (\ref{e36aa}), it is not difficult to show that
\begin{equation*}
\begin{split}
&||I_{r,1,1}^t(\cdot)-I_{r,1,2}^t(\cdot)||_{L^p}
\le \frac{C}{\nu^{\frac{1}{2}}(r-t)^{\frac{1}{2}}}\big(||v_1(T-r)-v_2(T-r)||_{L^{\infty}}
||g_1(T-r)||_{W^{1,p}}\\
&+||g_1(T-r)-g_2(T-r)||_{W^{1,p}}||v_2(T-r)||_{L^{\infty}}\big).
\end{split}
\end{equation*}
By Lemma \ref{l3},
\begin{equation*}
\begin{split}
&||I_{r,2,1}^t(\cdot)-I_{r,2,2}^t(\cdot)||_{L^p}
\le C\sup_{m=1,2}||v_m(T-r)||_{W^{2,p}}||v_1(T-r)-v_2(T-r)||_{W^{1,p}}
\end{split}
\end{equation*}
Noticing that $||v_m(t)||_{L^{\infty}}\le C||v_m(t)||_{W^{1,p}}$, combining all the estimates together
and using (\ref{e24}).
Hence there is some some $T_1>0$, if $0<T\le T_1$, then
\begin{equation}\label{e37a}
\begin{split}
&\sup_{t \in [0,T]}||\nabla g_1(t)-\nabla g_2(t)||_{L^p}\le
C||u_{0,1}-u_{0,2}||_{W^{1,p}}+CT^{\frac{1}{2}}\nu^{-\frac{1}{2}}\Big(
K_1\sup_{t \in [0,T]}||g_1(t)-g_2(t)||_{W^{1,p}}\\
&+\sup_{t \in [0,T]}||v_1(t)-v_2(t)||_{W^{1,p}}\sup_{t \in [0,T]}||g_1(t)||_{W^{1,p}}\Big)
+CT^{\frac{\beta-1}{\beta}}K_{2,\beta}\sup_{t \in [0,T]}||v_1(t)-v_2(t)||_{W^{1,p}}\\
&\le C||u_{0,1}-u_{0,2}||_{W^{1,p}}+CK_1T^{\frac{1}{2}}\nu^{-\frac{1}{2}}
\sup_{t \in [0,T]}||g_1(t)-g_2(t)||_{W^{1,p}}\\
&+
C(1+K_1^2+K_{2,\beta}^2)T^{\frac{\beta-1}{\beta}}(1+\nu^{-1})
\sup_{t \in [0,T]}||v_1(t)-v_2(t)||_{W^{1,p}}.
\end{split}
\end{equation}
Note that
\begin{equation*}
g_m(T-t,x)=\e\big(u_{0,m}(X_T^t(x))\big)-\int_t^T \e\big((v_m\nabla g_m)(T-r,X_r^t(x))\big)dr
+\int_t^T \e\big(F_{v_m}(T-r,X_r^t(x))\big)dr,
\end{equation*}
and, by   the same procedure as above, for every $0<T<T_1$,
\begin{equation}\label{e37aa}
\begin{split}
&\sup_{t \in [0,T]}||g_1(t)-g_2(t)||_{L^p}\le
C||u_{0,1}-u_{0,2}||_{L^p}+CTK_1\sup_{t \in [0,T]}||g_1(t)-g_2(t)||_{W^{1,p}}\\
&+C(1+K_1^2+K_{2,\beta}^2)T^{\frac{\beta-1}{\beta}}(1+\nu^{-1})
\sup_{t \in [0,T]}||v_1(t)-v_2(t)||_{W^{1,p}}.
\end{split}
\end{equation}
Based on (\ref{e37a}), (\ref{e37aa}), if we take a $T_0$ which only depends on
$K_1$, $K_{2,\beta}$, $\nu$, such that,
\begin{equation*}
CK_1T_0^{\frac{1}{2}}\nu^{-\frac{1}{2}}\le \frac{1}{4},\ \
CK_1T_0\le \frac{1}{4},
\end{equation*}
then the conclusion (\ref{e38a}) holds.
\end{proof}


For every $v \in \S(p,q,p',T)$ with
  some $d<p<\infty$, $1\le q \le\infty$, $0<T<1$, $1<p'<\frac{d}{2}$, $0<\alpha<1$, $\max(1,2-\alpha)<\beta<2$,
we can
define a map $\I'_{\nu}: \S(p,q,p',T) \rightarrow C([0,T];B_{p,q}^{1+\alpha}(\R^d;\R^d))$
$\bigcap L^{\beta}([0,T];B_{p,q}^{2+\alpha}(\R^d;\R^d)) $ by
$\I'_{\nu}(v)(t):=\mathbf{P}(Y_{T-t}^{T-t}(\cdot ))$, where $Y_s^t$ is the solution of (\ref{e29}) with coefficients
$v$ and initial condition $u_0=v(0)$, and $\mathbf{P}$ is the Leray-Hodge projection operator.

Analogously to Proposition \ref{p1} and Theorem \ref{t1}, we can show the following
results about the extension of $\I'_{\nu}$ and its fixed point.
\begin{thm}\label{t5.1}
Let $0<\alpha<1$, $d<p<\infty$, $1\le q <\infty$, $\max(1,2-\alpha)<\beta<2$. Suppose
$u_0 \in B_{p,q}^{1+\alpha}(\R^d;\R^d)$
satisfies $\nabla \cdot u_0=0$.
Then there exist $K_0>0$ and $0<T_0\le 1$
which only depend on $||u_0||_{B_{p,q}^{1+\alpha}}$,
$\nu$, such that
$\I'_{\nu}$ can be extended to be a map
$\I'_{\nu}: \B(u_0,T_0,p,q,\alpha,\beta,K_0)$ $\rightarrow \B(u_0,T_0,p,q,\alpha,\beta,K_0)$, where
\begin{equation}\label{e43aa}
\begin{split}
&\B(u_0,T_0,p,q,\alpha,\beta, K_0):=\Big\{v \in C([0,T_0];B_{p,q}^{1+\alpha}(\R^d;\R^d))
\bigcap L^{\beta}([0,T_0];B_{p,q}^{2+\alpha}(\R^d;\R^d));\\
&\ \ v(0,x)=u_0(x),\ \ ||v||_{1,T_0}\le K_0,\ \ ||v||_{2,T_0,\beta}\le K_0,\ \ \nabla \cdot v(t)=0,\ \forall t \in [0,T] \Big\},
\end{split}
\end{equation}
where $||v||_{1,T_0}:=\sup_{t \in [0,T_0]}||v(t)||_{B_{p,q}^{1+\alpha}}$ and
$||v||_{2,T_0,\beta}:=\big(\int_0^{T_0}||v(t)||_{B_{p,q}^{2+\alpha}}^{\beta}dt\big)^{\frac{1}{\beta}}$.
Moreover,
there exists a
constant $0<T_1<T_0$, which only depends on $||u_0||_{B_{p,q}^{1+\alpha}}$ and
the viscosity constant $\nu$, such that there is a unique fixed
point $u$ for the map $\I'_{\nu}$ in $\B(u_0,T_1,p,q,\alpha,\beta,K_0)$.
\end{thm}
\begin{proof}
Using the same procedure of the proof of Proposition \ref{p1}, for every $T>0$, $K>0$, and
$v\in \B(u_0,T,p,q,\alpha,\beta,K)$, there exists a sequence $\{v_n\}_{n=1}^{\infty}
\subseteq \S(p,q,p',T)$ for some $1<p'<\frac{d}{2}$, such that
\begin{equation}\label{e36a}
\begin{split}
&\lim_{n \rightarrow \infty}\sup_{t \in [0,T]}||v_n(t)-v(t)||_{W^{1,p}}=0,\\
&\sup_n \sup_{t \in [0,T]}||v_n(t)||_{B_{p,q}^{1+\alpha}}<\infty,\ \ \
\sup_n \int_0^T ||v_n(t)||_{B_{p,q}^{2+\alpha}}^{\beta}dt<\infty.
\end{split}
\end{equation}
By (\ref{e36a}) and Lemma \ref{l5.1}, we can find $K_0>>||u_0||_{B_{p,q}^{1+\alpha}}$, $0<T_0\le 1$
which only depend on $||u_0||_{B_{p,q}^{1+\alpha}}$ and $\nu$, such that for every
$v\in \B(u_0,T_0,p,q,\alpha,\beta,K_0)$,
\begin{equation}\label{e38aa}
\sup_n ||\I'_{\nu}(v_n)||_{1,T_0}\le K_0,\ \ \
\sup_n ||\I'_{\nu}(v_n)||_{2,T_0,\beta}\le K_0.
\end{equation}
According to Lemma \ref{l5.1a},
we know that $\{I(v_n)\}_{n=1}^{\infty}$ is a Cauchy sequence
in $C([0,T_0];W^{1,p}$ $(\R^d;$ $\R^d))$, so it has a limit $\tilde v \in C([0,T_0];W^{1,p}(\R^d;$ $\R^d))$.
In particular, such limit $\tilde v$ is independent of the choice of sequence $\{v_n\}$, we define
$\I'_{\nu}(v):=\tilde v$.

By (\ref{e38aa}),
as the same procedure in the proof of Proposition \ref{p1}, we can show $\tilde v \in C([0,T_0];B_{p,q}^{1+\alpha}(\R^d;\R^d))$
 $\bigcap L^{\beta}([0,T_0];B_{p,q}^{2+\alpha}(\R^d;\R^d))$. In particular, in order to prove the associated estimate
 (\ref{e28aa}) for $B_{p,q}^{\alpha}$ norm, $C_c^{\infty}(\R^d)$ need to be dense in
 $B_{p,q}^{\alpha}(\R^d)$, so the case $q=\infty$ can not be included.

Based on (\ref{e36a}), replacing $B_{p,p}^{r'}$ norm in (\ref{e23a}) by $W^{1,p}$ norm, and repeating the proof
of Theorem \ref{t1}, we can show that there is a constant $0<T_1<T_0$ which only depends on
$||u_0||_{B_{p,q}^{1+\alpha}}$ and $\nu$, such that
there exists a unique fixed point for $\I'_{\nu}$ in $\B(u_0,T_1,p,q,\alpha,\beta,K_0)$.
\end{proof}

\begin{rem}
In contrast with Theorem \ref{t1}, the local existence time $T_0$ for the fixed point
$\I'_{\nu}$ depends on the viscosity constant $\nu$, which is due to the dependence
of $\nu$ in the estimate (\ref{e24}).
\end{rem}

Repeating the proof of Theorem \ref{t2}, we can also show that the fixed point
$u$ is a solution of (\ref{e1}).

\begin{thm}\label{t5.2}
Let  $0<\alpha<1$, $d<p<\infty$, $1\le q <\infty$, $\max(1,2-\alpha)<\beta<2$.
Suppose that  $u_0 \in  B_{p,q}^{1+\alpha}(\R^d;\R^d)$ which satisfies $\nabla \cdot u_0=0$.
Then
we can find a constant $T_0>0$ which only depends
$||u_0||_{B_{p,q}^{1+\alpha}}$ and $\nu$, such that
there exists a vector field
$u \in C([0,T_0];B_{p,q}^{1+\alpha}(\R^d;\R^d))$
$\bigcap L^{\beta}([0,T_0];B_{p,q}^{2+\alpha}(\R^d;\R^d))$
which is the unique strong solution of (\ref{e1}) in the space
$u \in C([0,T_0];B_{p,q}^{1+\alpha}(\R^d;\R^d))$
$\bigcap L^{\beta}([0,T_0];B_{p,q}^{2+\alpha}(\R^d;\R^d))$.



\end{thm}

\begin{rem}
Note that in Theorem \ref{t5.2} the uniqueness of solution
needs to hold in a subspace of  $C([0,T];B_{p,q}^{1+\alpha}(\R^d;\R^d))$,
i.e.,  $C([0,T];B_{p,q}^{1+\alpha}(\R^d;\R^d))$
$\bigcap L^{\beta}([0,T];B_{p,q}^{2+\alpha}(\R^d;\R^d))$, since
we have to control the norm $||\nabla v(t)||_{L^{\infty}}$
in the iteration procedure.

Also note that for $p>1$, $r>1+\frac{d}{p}$, $||\nabla v(t)||_{L^{\infty}}
\le C||v(t)||_{B_{p,q}^r}$, in the similar way as above we can show the
following local existence of a unique solution of (\ref{e1}) in lower and
higher order Besov space.

Suppose  $1<p<\infty$, $1\le q <\infty$, $r>\max(1,\frac{d}{p})$,
$u_0 \in  B_{p,q}^{r}(\R^d;\R^d)$ satisfying that $\nabla \cdot u_0=0$.
Then
we can find a constant $T_0>0$ which only depends
$||u_0||_{B_{p,q}^{r}}$ and $\nu$, such that
there exists a vector field
$u \in C([0,T_0];B_{p,q}^{r}(\R^d;\R^d))$
$\bigcap L^{\beta}([0,T_0];B_{p,q}^{r+1}(\R^d;\R^d))$ for some
$1<\beta<2$,
which is the unique strong solution for (\ref{e1}) in the space
$u \in C([0,T_0];B_{p,q}^{r}(\R^d;\R^d))$
$\bigcap L^{\beta}([0,T_0];B_{p,q}^{r+1}(\R^d;\R^d))$.
Moreover, if $r>1+\frac{d}{p}$, the local unique existence of strong
solution for (\ref{e1}) holds for $u \in C([0,T_0];B_{p,q}^{r}(\R^d;\R^d))$.
\end{rem}

\begin{rem}
Tracking the proof of Proposition \ref{p1}, we only need the condition that
$C_c^{\infty}(\R^d)$ is dense in $B_{p,q}^{\alpha}(\R^d)$ to show $\I'_{\nu}(v)$
is continuous with the time parameter under $B_{p,q}^{1+\alpha}$ norm. Hence if we consider
the $B_{p,\infty}^r$ norm for $p>1$, $r>\max(1,\frac{d}{p})$,
the local existence result can be derived with the function space
$C([0,T_0];B_{p,\infty}^{r}(\R^d;\R^d))$ $\bigcap L^{\beta}([0,T_0];B_{p,q}^{r+1}(\R^d;\R^d))$
replaced by $L^{\infty}([0,T_0];B_{p,\infty}^{r}(\R^d;\R^d))$
$\bigcap C([0,T_0];B_{p,\infty}^{s}(\R^d;\R^d))$ $\bigcap L^{\beta}([0,T_0];B_{p,\infty}^{r+1}(\R^d;\R^d))$
with any $0<s<r$ and some $1<\beta<2$.
\end{rem}

\vskip 5mm

\bf Acknowledgments:
\rm

The first two authors aknowledge financial support from the FCT project

PTDC/MAT/104173/2008.

\end{document}